\documentclass[preprint,12pt]{elsarticle}

\usepackage{amssymb}
\usepackage{amsfonts}
\usepackage{graphicx}
\usepackage{epsfig}
\usepackage{geometry}
\usepackage{tikz}
\usetikzlibrary{trees,positioning,shapes}
\usepackage{subcaption}
\usepackage{amsmath}
\usepackage{amsthm}

\usepackage{xcolor, colortbl}

\usepackage{algorithm}
\usepackage{algorithmicx, algpseudocode}

\usepackage{stmaryrd}

\biboptions{sort&compress}

\newcommand{\revision}[1]{\textcolor{black}{#1}}

\newcommand{\inblue}[1]{\textcolor{black}{#1}}

\newtheorem{theorem}{Theorem}

\newtheorem{remark}{Remark}

\newcommand{\Prod}[2]{(#1, #2)_{L^2}}

\newcommand{\Vect}[1]{\mathbf{#1}}

\newcommand{\Spl}{\mathcal{S}}
\newcommand{\SplO}{\widetilde{\mathcal{S}}}

\newcommand{\HH}{\mathbf{H}^1_0(\Omega)}

\newcommand{\U}{\Vect{u}}
\newcommand{\F}{\Vect{f}}

\newcommand{\V}{\Vect{v}}
\newcommand{\Div}{\nabla\cdot}

\newcommand{\grad}{\nabla}

\DeclareMathOperator*{\argmin}{arg\,min}

\begin{document}

\begin{frontmatter}

\title{Isogeometric Residual Minimization (iGRM) for \revision{Non-Stationary Stokes and Navier-Stokes} Problems}
\author{M. \L{}o\'{s}$^1$, I. Muga$^2$, J. Mu{\~n}oz-Matute$^3$, M. Paszy\'{n}ski$^1$ }

\address{$^1$Department of Computer Science, \\ AGH University of Science and Technology,
Krakow, Poland \\
e-mail: marcin.los.91@gmail.com\\
e-mail: paszynsk@agh.edu.pl 
}

\address{$^{\textrm{2}}$  Instituto de Matem\'aticas, Pontificia Universidad Cat\'olica de Valpara\'iso, Chile\\
e-mail: ignacio.muga@pucv.cl}

\address{$^3$The University of Basque Country, Bilbao, Spain \\
e-mail: judith.munozmatute@gmail.com}

\begin{abstract}
\revision{We show that it is possible to obtain a linear computational cost FEM-based solver for non-stationary Stokes and Navier-Stokes equations. Our method employs a technique developed by Guermond \& Minev \cite{Minev}, which consists of singular perturbation plus a splitting scheme. While the time-integration schemes are implicit, we use finite elements to discretize the spatial counterparts. At each time-step, we solve a PDE having weak-derivatives in one direction only (which allows for the linear computational cost), at the expense of handling strong second-order derivatives of the previous time step solution, on the right-hand side of these PDEs. This motivates the use of smooth functions such as B-splines. For high Reynolds numbers, some of these PDEs become unstable. To deal robustly with these instabilities, we propose to use a residual minimization technique. We test our method on problems having manufactured solutions, as well as on the cavity flow problem.}
\end{abstract}
	

\end{frontmatter}

\section{Introduction}

In this paper, we focus on \revision{non-stationary Stokes and Navier-Stokes problems, which} requires special stabilization effort, for large Reynolds numbers. There are multiple stabilization techniques developed for standard finite element methods \cite{s1,s2,s3,s4,s5}. 
Here, we apply a residual minimization approach for stabilization (see, e.g.,~\cite{evans,VSM}), together with B-spline basis functions from isogeometric analysis (IGA) for the discretization in space.

Minimum residual methods aim to find $u_h\in U_h$ such that
$$
u_h=\arg\!\!\!\min_{w_h\in U_h}\|b(w_h,\cdot) - \ell(\cdot)\|_{V^*}\,,
$$
where $U$ and $V$ are Hilbert spaces, $b:U\times V\to \mathbb R$ is a continuous bilinear (weak) form, $U_h\subset U$ is a discrete trial space, and $\ell\in V^*$ (the dual space of $V$) is a given right-hand side. Several discretization techniques are particular incarnations of this wide-class of residual minimization methods. These include: the \emph{least-squares finite element method}~\cite{LSFEM}, the \emph{discontinuous Petrov-Galerkin method (DPG) with optimal test functions}~\cite{t8}, or the \emph{variational stabilization method}~\cite{VSM}.
We approach the residual minimization \inblue{method} using its saddle point (mixed) formulation, e.g., as described in~\cite{evans}.  

We employ the splitting scheme from Jean-Luc Guermond and Petar Minev \cite{Minev} to express \revision{the non-stationary Stokes or }Navier-Stokes problem as a sequence of pressure predictor, velocity update, and pressure corrector. Each of these steps can be solved in a linear computational cost ${\cal O}(N)$ using the Kronecker product structure of the underlying matrices.
\revision{Instead of using finite differences, we use finite elements for space discretizations.}

The method presented in this paper is an extension of the linear computational cost Kronecker product solver that we proposed for the advection-diffusion problem \cite{Judit}. The algorithm we propose uses the method of lines (discretization in space with time iterations) and delivers a linear computational cost ${\cal O}(N)$ solver for each time step. In this sense, it is an alternative to the space-time formulation \cite{Sangiali}, where an iterative solver is employed. \inblue{In that configuration,} the size of the mesh is equal to $M=N\times k$ \inblue{(for the uniform case)}, where $k$ is the number of time steps. Hence, if $c$ is the number of iterations of the solver, its total cost becomes $M\times c=N\times k\times c$, while our cost is just $N\times k$. However, the space-time formulation allows for adaptation in both space and time, where some parts of the space-time mesh can use ``larger'' time steps \inblue{than} the others. Thus, the case of space-time adaptivity can be indeed competitive to our method. We also differ from \cite{Sangiali} in the sense that we deliver automatic stabilization, with a linear computational cost.

Our paper is also different from \cite{Minev}. Even though they use finite differences discretization with a linear computational cost direction splitting algorithm, they do not incorporate residual minimization stabilization. On the other hand, we discretize the space with B-splines over the patch of elements, and we use a weak form of the splitting scheme. We add the residual minimization procedure in a way that preserves the linear computational cost of the solver. Our method delivers high continuity approximations and automatic stabilization in every time step.

The residual minimization method used here is very similar to the DPG method developed by \cite{DPGStokes}. Our spaces, trial, and test are continuous, while they are broken in the DPG method. The motivation behind breaking the spaces is to obtain a block-diagonal structure of the matrix. Thus, static condensation practically eliminates the inner product matrix, leaving only fluxes or traces over the edges between elements. However, breaking the spaces makes the linear cost factorization impossible since it destroys the Kronecker product structure of the matrix. So does the mesh adaptation. For the DPG method, there are some modern multi-grid solvers developed, allowing for fast factorizations of the system \cite{multigrid}.

\revision{We show that the Navier-Stokes simulation using the alternating directions implicit solver without the residual minimization stabilization generates some unexpected oscillations for high Reynolds number.
In our paper, we use the splitting scheme from Jean-Luc Guermond and Petar Minev \cite{Minev}, which translates the Navier-Stokes problem into the velocity updates involving reaction-dominated diffusion kind of equation, followed by the explicit pressure updates.
Instead of using finite differences, we use standard and higher-order finite elements for space discretizations.
This kind of setup is unstable for high Reynolds number in a similar manner how the reaction-dominated diffusion equations are unstable for a very small diffusion coefficient. We show how the stabilization for high Reynolds number of the Navier-Stokes formulation is performed by the residual minimization method.}

\revision{The novel contributions of our paper with respect to \cite{r1,r2} are the following. In \cite{r1}, the authors propose and study Stokes problem formulation with IGA discretization. They use the Taylor-Hood isogeometric element, and the NURBS basis functions. They augment the paper with 2D and 3D numerical results for the Stokes flow problem. In \cite{r2}, the authors introduce the Raviar-Thomas, Nedelec, and Taylor-Hood elements within the IGA formulation. They test the numerical properties of these formulations on the 2D Stokes problem using simple NURBS geometries.}

\revision{In summary, papers \cite{r1,r2} deal with the stationary Stokes problem, while the main focus of our paper is on the non-stationary Stokes a Navier-Stokes problem.
Our paper shows that for the implicit time integration scheme with non-stationary Stokes or Navier-Stokes equations, it is possible to perform direction splitting, including the residual minimization method. The computational cost of the resulting direct solver is linear. Our paper's main scientific contribution is the linear computational cost solver for non-stationary Stokes and Navier-Stokes problems, stabilized with the residual minimization method. We perform numerical experiments for several trial and test spaces.}

\revision{The common understanding is that the direction splitting solver is limited to tensor product grids. It has been shown in \cite{r4,r5} that this solver can be extended to complicated geometries for finite difference simulations. The investigation on the possibility of applying this method for complicated geometries in the context of the finite element method (classical and isogeometric) will be the subject of our future work, either by building iterative solvers or using similar tricks as presented in \cite{r4,r5}.}

\revision{In our paper, we advocate linear computational cost alternating direction, implicit solver, on a tensor product grids. We show that higher continuity spaces for this kind of solvers reduce the computational cost of the simulations, maintaining high numerical accuracy.
However, the standard finite element spaces can be augmented with $hp$-adaptivity, as presented in DPG framework \cite{t8}, resulting in exponential convergence of the numerical error with respect to the mesh size. In this paper, we advocate an alternative approach, and we are aware that higher continuity spaces are not suitable for $hp$ adaptivity.}

The structure of this paper is as follows. \inblue{We describe in Section 2 the residual minimization method in an abstract way, and we introduce the functional settings that we use in the paper. Section 3 introduces the non-stationary Stokes and Navier-Stokes problems, together with alternating direction splitting and residual minimization methods. Finally, Section 4 presents numerical results using manufactured solutions and the cavity flow problem. We conclude the paper in Section 5.}

\section{Abstract settings}
\subsection{The residual minimization method}
\inblue{Let $U$ and $V$ be Hilbert spaces and let $b:U\times V\to \mathbb R$ be a continuous bilinear form. Given $\ell\in V^*$\footnote{$V^*$ denotes the dual space of $V$, i.e., the space of linear and continuous functionals over $V$.}, an abstraction of the variational problems that we will find throughout this paper is:} 
\begin{equation}
\left\{
\begin{array}{l}
\text{Find $u\in U$ such that}\\
b(u,v)=\ell(v), \quad \forall v\in V.
\end{array}\right.
\label{eq:gen_weak}
\end{equation}
\inblue{We assume that problem~\eqref{eq:gen_weak} is well-posed in the sense that for any $\ell\in V^*$, there exists a unique solution $u\in U$, and $\|u\|_U\lesssim\|\ell\|_{V^*}=\|b(u,\cdot)\|_{V^*}$ with a (stability) constant that only depends on $b(\cdot,\cdot)$. Let $U_h\subset U$ be a discrete subspace where we want to approximate the solution of~\eqref{eq:gen_weak}. If $u_h\in U_h$ is such an approximation, we define its correspondent residual as the expression $\ell(\cdot)-b(u_h,\cdot)\in V^*$. Notice that the exact solution of~\eqref{eq:gen_weak} has zero residual. Therefore, the residual minimization method aims to find $u_h\in U_h$ minimizing the $V^*$-norm of the correspondent residual, i.e., 
\begin{equation}\label{eq:ResMin}
u_h=\argmin_{w_h\in U_h}\|\ell(\cdot)-b(w_h,\cdot)\|_{V^*}.
\end{equation}
It is well-known that problem~\eqref{eq:ResMin} has a unique solution (we are minimizing a strictly convex functional over a convex set, and the map $w_h\mapsto b(w_h,\cdot)\in V^*$ is injective). Unfortunately, the computation of the dual norm $\|\cdot\|_{V^*}$ requieres to solve the Riesz map inversion:\footnote{Indeed, $\|\ell(\cdot)-b(u_h,\cdot)\|_{V^*}=\|r\|_V$.}
\begin{equation}
\left\{
\begin{array}{l}
\text{Find $r\in V$ such that}\\
(r,v)_V =\ell(v) -   b(u_h,v), \quad \forall v\in V,
\end{array}\right.
\label{eq:Riesz}
\end{equation}
where $(\cdot,\cdot)_V$ denotes the inner product of the Hilbert space $V$. Such a computation is unfeasible unless $V$ has a finite dimension. Since this is not the general situation, we must consider a discrete counterpart $V_h\subset V$, which leads to the discrete residual minimization problem:\footnote{For the sake of simplicity, we have kept the same notation for the minimizer $u_h\in U_h$.}
\begin{equation}\label{eq:dResMin}
u_h=\argmin_{w_h\in U_h}\|\ell(\cdot)-b(w_h,\cdot)\|_{(V_h)^*}.
\end{equation}
However, problem~\eqref{eq:dResMin} may not be unique now, since two different elements of $U_h$, say $w_h$ and $\tilde w_h$, may satisfy $b(w_h,\cdot)=b(\tilde w_h,\cdot)$ in $(V_h)^*$. To avoid this drawback a \emph{discrete inf-sup condition} is needed, i.e., $V_h$ must be such that:
\begin{equation}\label{eq:dInf-Sup}
\|b(w_h,\cdot)\|_{(V_h)^*}=\sup_{v_h\in V_h}{b(w_h,v_h)\over \|v_h\|_V}\geq \gamma_h \|w_h\|_U, \quad\forall w_h\in U_h,
\end{equation}
for some \emph{discrete inf-sup} (stability) constant $\gamma_h>0$.
\begin{remark}
A space $V_h\subset V$ satisfying~\eqref{eq:dInf-Sup} always exists. Indeed, the theory of optimal test functions developed by Demkowicz \& Gopalakrishnan~\cite{t8} establishes the existence of an optimal test space $V_h^\text{opt}\subset V$, of the same dimension of $U_h$, satisfying~\eqref{eq:dInf-Sup} with the same stability constant of the continuous problem~\eqref{eq:gen_weak} (cf.~\cite{Schaback}). Notice that, contrary to standard \revision{Petrov-Galerkin} methods, there is no restriction here about the dimension of $V_h$ needed to satisfy~\eqref{eq:dInf-Sup}.
\end{remark}
Let us focus now on the practical way to compute~\eqref{eq:dResMin}. The standard way to do so is to solve the mixed (saddle point) problem to find $(r_h,u_h)\in V_h\times U_h$ such that (cf.~\cite{evans,VSM}):
\begin{equation}\label{eq:mixed}
\left\{
\begin{array}{lll}
(r_h,v_h)_V+b(u_h,v_h) & = \ell(v_h), & \forall v_h\in V_h,\\
b(w_h,r_h) & = 0, & \forall w_h\in U_h.
\end{array}\right.
\end{equation}
Notice that the first equation in~\eqref{eq:mixed} computes the residual representative (same as in eq.~\eqref{eq:Riesz}), while the second equation establishes the optimality condition, i.e., the residual representative $r_h\in V_h$ has to be orthogonal to the approximation space $\{b(w_h,\cdot)\in (V_h)^*: w_h\in U_h\}$. We end up this section with a theorem that establishes the well-posedness of system~\eqref{eq:mixed}, as well as {\it a priori} estimates.
\begin{theorem}
If $U_h\subset U$ and $V_h\subset V$ are discrete spaces such that the discrete inf-sup condition~\eqref{eq:dInf-Sup} is satisfied, then the mixed problem~\eqref{eq:mixed} has a unique solution $(r_h,u_h)\in V_h\times U_h$. Moreover,
\begin{alignat}{2}
\|r_h\|_V\leq & \|\ell\|_{V^*}\label{eq:apriori1}\\
\|u_h\|_U\leq & {1\over \gamma_h}\|\ell\|_{V^*}= {1\over \gamma_h}\|b(u,\cdot)\|_{V^*}
\leq{M_b\over\gamma_h}\|u\|_U\label{eq:apriori2}\\
\|u-u_h\|_U\leq & {M_b\over\gamma_h} \inf_{w_h\in U_h}\|u-w_h\|_U, \label{eq:apriori3}
\end{alignat}
where $u\in U$ is the exact solution of problem~\eqref{eq:gen_weak} and $M_b>0$ is such that $|b(w,v)|\leq M_b\|w\|_U\|v\|_V$, for any $w\in U$ and $v\in V$. 
\end{theorem}
\begin{proof}
\revision{See, e.g.,~\cite[Propositions 2.2 and 2.3]{BroSteCAMWA2014}.}
\end{proof}
}

\subsection{Functional spaces}\label{sec:functional}
\inblue{
Let $\Omega\subset\mathbb R^2$ be a bounded Lipschitz domain and let $L^2(\Omega)$ be the Hilbert space of square integrable functions over $\Omega$, endowed with its standard inner product $(\cdot,\cdot)_{L^2}$. We will use bold fonts to denote vector fields. Likewise, if $\F=(f_1,f_2)$ and  $\U=(u_1,u_2)$ are functions in ${\mathbf L}^2(\Omega):=L^2(\Omega)\times L^2(\Omega)$, then 
$(\F,\U)_{L^2}:=(f_1,u_1)_{L^2}+(f_2,u_2)_{L^2}$.
The notation $L_0^2(\Omega)$ will refer to the closed subspace of $L^2(
\Omega)$ of functions with zero mean, while $H_0^1(\Omega)$ will denote the Sobolev space of square integrable functions having square integrable first-oder weak-derivatives and zero trace over the boundary $\partial\Omega$. Moreover, we consider the space $\HH: = H^1_0(\Omega) \times H^1_0(\Omega)$ endowed with the semi-inner product 
$\Prod{\nabla\V}{\nabla \U} := \Prod{\nabla v_1}{\nabla u_1} + \Prod{\nabla v_2}{\nabla u_2}$,
for any $\V,\U\in\HH$.}

\inblue{
For direction splitting purposes, we will need one-way-derivative spaces. Thus, given a unitary vector $\boldsymbol\beta\in\mathbb R^2$, we define the graph space $V_{\boldsymbol\beta}:=\{u\in L^2(\Omega): \boldsymbol\beta\cdot\nabla u\in L^2(\Omega)\}$, endowed with the inner product $(v,u)_{V_{\boldsymbol\beta}}:=(v,u)_{L^2}+(\boldsymbol\beta\cdot\nabla v,\boldsymbol\beta\cdot\nabla u)_{L^2}$. It is well-known (see, e.g.,~\cite{DiPietroErn}) that this graph space has well-defined traces within the space $L^2(|{\boldsymbol\beta\cdot\boldsymbol n}|;\partial\Omega):=\{u \text{ measurable on }\partial\Omega: \int_{\partial\Omega}|{\boldsymbol\beta\cdot\boldsymbol n}|u^2<+\infty\}$, where $\boldsymbol n$ denotes the outer-poiting normal vector. Let $V_{\boldsymbol\beta,0}:=\{u\in V_{\boldsymbol\beta}:({\boldsymbol\beta\cdot\boldsymbol n})u=0 \text{ over } \partial\Omega\}$ the closed subspace of zero trace functions. For $\boldsymbol \beta=\boldsymbol i=(1,0)$ and $\boldsymbol \beta=\boldsymbol j=(0,1)$, we will use the spaces
$V_{\boldsymbol i}$ and $V_{\boldsymbol j}$, together with the spaces $\mathbf V_{\boldsymbol i,0}:=V_{\boldsymbol i,0}\times V_{\boldsymbol i,0}$ and  $\mathbf V_{\boldsymbol j,0}:=V_{\boldsymbol j,0}\times V_{\boldsymbol j,0}$ (see Section~\ref{sec:N-S}).
}

\inblue{
Our discrete functional spaces will be based on the space $\Spl^{m,k}$ of 2D B-splines of degree~$m\in\mathbb N$ and global continuity~$k\in\mathbb N$. Accordingly,
$\Spl^{m,k}_0 \subset \Spl^{m,k}$ will denote the subspace without boundary degrees of freedom (i.e., $\Spl^{m,k}_0 = \Spl^{m,k} \cap H^1_0(\Omega)$), and
$\SplO^{m,k}\subset\Spl^{m,k}$ will denote the subspace where only one degree of freedom  has been removed from the boundary. }

\section{Time-dependent Stokes and Navier-Stokes problems}\label{sec:N-S}

Let \inblue{$\Omega=\Omega_x\times\Omega_y$} be a square domain and $I=(0,T]\subset\mathbb{R}$. We consider the following two-dimensional Navier-Stokes equation:
\begin{equation}\label{strong}
\displaystyle{ \left\{
\begin{split}
\partial_{t}\mathbf{v}+\inblue{(\mathbf{v}\cdot \grad) \mathbf{v}-\frac{1}{R_e}}\Delta\mathbf{v}+\nabla p=&\;\mathbf{f}&\mbox{in}&\;\Omega\times I,\\
\grad\cdot\mathbf{v}=&\;0&\mbox{in}&\;\Omega\times I,\\
\mathbf{v}=&\;0&\mbox{in}&\;\Gamma\times I,\\
\mathbf{v}(0)=&\;\mathbf{v}_{0}\;&\mbox{in}&\;\Omega,\\
\end{split}
\right.} 
\end{equation}
where $\mathbf{v}=(v_{1},v_{2})$ is the unknown velocity field and $p$ is the unknown pressure. Here, $\Gamma=\Gamma_{x}\cup\Gamma_{y}$ denotes the boundary of the spatial domain $\Omega$, $\mathbf{f}$ is a given source, $\mathbf{v}_{0}$ is a given initial condition, and
\inblue{$R_e$ stands for the dimensionless Reynolds number. When we drop the advection term $(\mathbf{v} \cdot \grad) \mathbf{v}$, we obtain the non-stationary Stokes problem, and for this case, the notion of Reynolds number does not make sense since there is no physical concept of Reynolds number in Stokes flows. All Stokes flows with different diffusion constants are self-similar under a simple rescaling of time and pressure, or alternatively, rescaling the domain and pressure. Thus, in the numerical results section, when we consider the non-stationary Stokes problem, we remove the Reynolds number from the equations.}

Following \cite{Minev}, we consider the following singular perturbation of problem (\ref{strong}):
\begin{equation}\label{perturbed}
\displaystyle{ \left\{
\begin{split}
\partial_{t}\mathbf{v}_{\epsilon}\inblue{+(\mathbf{v}_{\epsilon}\cdot \grad) \mathbf{v}_{\epsilon}-{1\over{R_e}}}\Delta\mathbf{v}_{\epsilon}+\nabla p_{\epsilon}=&\;\mathbf{f}&\mbox{in}&\;\Omega\times I,\\
\epsilon A\phi_{\epsilon}+\grad\cdot\mathbf{v}_{\epsilon}=&\;0&\mbox{in}&\;\Omega\times I,\\
\epsilon\partial_{t}p_{\epsilon}=\;\phi_{\epsilon}-\chi\inblue{\frac{1}{R_e}}\grad\cdot& \mathbf{v}_{\epsilon}&\mbox{in}&\;\Omega\times I,\\
\mathbf{v}_{\epsilon}=&\;0&\mbox{in}&\;\Gamma\times I,\\
\mathbf{v}_{\epsilon}(0)=&\;\mathbf{v}_{0}\;&\mbox{in}&\;\Omega,\\
p_{\epsilon}(0)=&\;p_{0}\;&\mbox{in}&\;\Omega,\\
\end{split}
\right.} 
\end{equation}
where $A: D(A)\subset L^{2}_{0}(\Omega)\longrightarrow L^{2}_{0}(\Omega)$ is an unbounded operator, $\phi_{\epsilon}\in D(A)$, $\epsilon$ is the perturbation parameter and $\chi\in[0,1]$ is a user defined parameter. 
\inblue{As in \cite{Minev}, we are going to use implicit time-integration schemes, while the non-linear advection term $(\mathbf{v}_{\epsilon}\cdot \grad) \mathbf{v}_{\epsilon} $ is treated explicitly. }

\subsection{Alternating Direction Implicit (ADI) method}
We consider the ADI method presented in \cite{Minev}, with the Peaceman-Rarchford scheme \cite{ADS1,ADS2} applied to the velocity update. First, we perform an uniform partition of the time interval $\bar{I}=[0,T]$, \inblue{i.e., 
$0=t_{0}<t_{1}<\ldots<t_{N-1}<t_{N}=T$, where $\tau={T\over N}$ denotes the uniform time-step.}
In (\ref{perturbed}), we select $\epsilon=\tau$ and $A=(1-\partial^2_{x})(1-\partial^2_{y})$. The scheme reads as follows:
\begin{itemize}
\item \textit{Pressure predictor.} 
We set $\tilde{p}^{n+\frac{1}{2}}=p^{n-\frac{1}{2}}+\phi^{n-\frac{1}{2}},\;\forall n=0,\ldots,N-1$ being $p^{-\frac{1}{2}}=p_{0}$ and $\phi^{-\frac{1}{2}}=0$.
\item \textit{Velocity update.} We set $\mathbf{v}^{0}=\mathbf{v}_{0}$ and solve:
\begin{equation}\label{ADIv1}
\left\{
\begin{array}{rl}
\displaystyle\left(1-\frac{\tau\partial^2_{x}}{2R_e}\right)\mathbf{v}^{n+\frac{1}{2}}
= & \displaystyle\left(1+\frac{\tau\partial^2_{y}}{2R_e}\inblue{-(\mathbf{v}^{n}\cdot\grad)}\right)\mathbf{v}^{n}
-\frac{\tau}{2}\grad\tilde{p}^{n+\frac{1}{2}}+\frac{\tau}{2}\mathbf{f}^{n+\frac{1}{2}}
,\\
\mathbf{v}^{n+\frac{1}{2}}= & 0\quad \mbox{in }\Gamma_{x}.
\end{array}\right.
\end{equation}
\begin{equation}\label{ADIv2}
\left\{
\begin{array}{rl}
\displaystyle \left(1-\frac{\tau\partial^2_{y}}{2R_e}\right)\mathbf{v}^{n+1}
= &\displaystyle \left(1+\frac{\tau\partial^2_{x}}{2R_e}\right)\mathbf{v}^{n+\frac{1}{2}}
-\frac{\tau}{2}\grad\tilde{p}^{n+\frac{1}{2}}+\frac{\tau}{2}\mathbf{f}^{n+\frac{1}{2}},\\
\mathbf{v}^{n+1}= & 0\quad\mbox{in }\Gamma_{y}.
\end{array}\right.
\end{equation}
\item \textit{Penalty step.}\\
\begin{equation}\label{ADIpp}
\left\{
\begin{array}{rl}
(1-\partial^2_x)\psi
= & -\displaystyle\frac{1}{\tau}\grad\cdot\mathbf{v}^{n+1},\\
\partial_{x}\psi= & 0\quad\mbox{in }\Gamma_{x}.
\end{array}\right.
\quad\text{ and}\quad
\left\{
\begin{array}{rl}
(1-\partial^2_y)\phi^{n+\frac{1}{2}}
= & \psi,\\
\partial_{y}\phi^{n+\frac{1}{2}}= & 0\quad\mbox{in }\Gamma_{y}.
\end{array}\right.
\end{equation}
\item \textit{Pressure update.}\\
\begin{equation}\label{ADIp}
p^{n+\frac{1}{2}}=p^{n-\frac{1}{2}}+\phi^{n+\frac{1}{2}}-\chi \inblue{\frac{1}{R_e}} \grad\cdot\left(\frac{1}{2}(\mathbf{v}^{n+1}+\mathbf{v}^{n})\right).
\end{equation}
\end{itemize}

\subsection{Variational formulations}
\inblue{Recall the graph spaces $V_{\boldsymbol\beta}$ and $\mathbf V_{\boldsymbol\beta,0}$ from Section~\ref{sec:functional}, and denote by $(\cdot,\cdot)$ both $\mathbf{L}^{2}(\Omega)$ and $L^{2}(\Omega)$ inner products. Let $\boldsymbol i=(1,0)$ and $\boldsymbol j=(0,1)$. Equation~\eqref{ADIv1} admits the following variational formulation:}
\begin{equation}\label{testADIv1}
\left\{
\begin{array}{l}
\text{Find } \mathbf{v}^{n+\frac{1}{2}}\in\mathbf V_{\boldsymbol i,0} \text{ such that}\\
\displaystyle(\mathbf{v}^{n+\frac{1}{2}},\mathbf{u})+\tau{(\partial_{x}\mathbf{v}^{n+\frac{1}{2}},\partial_{x}\mathbf{u})\over2R_e}= \left(\mathbf{v}^{n}+\frac{\tau}{2}\left[{\partial^2_{y}\mathbf{v}^{n}\over R_e}-(\mathbf{v}^{n}\cdot\grad) \mathbf{v}^{n}-\grad\tilde{p}^{n+\frac{1}{2}}+\mathbf{f}^{n+\frac{1}{2}}\right],\mathbf{u}\right)
\\
\forall\,\mathbf u\in \mathbf V_{\boldsymbol i,0}.
\end{array}\right.
\end{equation}
Equation~\eqref{ADIv2} admits the following variational formulation:
\begin{equation}\label{testADIv2}
\left\{
\begin{array}{l}
\text{Find } \mathbf{v}^{n+1}\in\mathbf V_{\boldsymbol j,0} \text{ such that}\\
\displaystyle(\mathbf{v}^{n+1},\mathbf{u})+\tau{(\partial_{y}\mathbf{v}^{n+1},\partial_{y}\mathbf{u})\over 2R_e}= \left(\mathbf{v}^{n+\frac{1}{2}}+\frac{\tau}{2}\left[{\partial^2_{x}\mathbf{v}^{n+\frac{1}{2}}\over R_e}-\grad\tilde{p}^{n+\frac{1}{2}}+\mathbf{f}^{n+\frac{1}{2}}\right],\mathbf{u}\right)\\
\forall\,\mathbf u\in \mathbf V_{\boldsymbol j,0}.
\end{array}\right.
\end{equation}
Finally, equations~\eqref{ADIpp} admit the following variational formulations:
\begin{equation}\label{testADIpsi}
\left\{
\begin{array}{l}
\text{Find } \psi\in V_{\boldsymbol i} \text{ such that}\\
\displaystyle(\psi,w)+(\partial_{x}\psi,\partial_{x}w)=-\frac{1}{\tau}(\grad\cdot\mathbf{v}^{n+1},w),
\quad\forall w\in V_{\boldsymbol i}.
\end{array}\right.
\end{equation}
and
\begin{equation}\label{testADIphi}
\left\{
\begin{array}{l}
\text{Find }\phi^{n+\frac{1}{2}}\in V_{\boldsymbol j} \text{ such that}\\
(\phi^{n+\frac{1}{2}},w)+(\partial_{y}\phi^{n+\frac{1}{2}},\partial_{y}w)=(\psi,w),
\quad\forall w\in V_{\boldsymbol j}.
\end{array}\right.
\end{equation}
\inblue{
Notice that all the previous variational formulations are coercive in their respective graph space. 
The presented \revision{singular perturbation} and the splitting scheme will lead to a set of well-posed equations with just Galerkin discretization, without needing the more expensive iGRM discretization. However, high Reynolds numbers $R_e>0$ translate into a loss of stability in formulations~\eqref{testADIv1} and~\eqref{testADIv2}. So the true value of the iGRM stabilization appears when $R_e>>1$}.
\inblue{
\subsection{iGRM stabilization}
We show here how to discretize the previous variational formulations in the sense of residual minimization. As usual, we select tensor product B-splines basis functions. In formulations~\eqref{testADIv1} and~\eqref{testADIv2} we use $\HH$-conforming basis functions to make sure that the boundary condition in~\eqref{perturbed} is satisfied over the whole boundary $\Gamma$ at any time-step.  Let $M^x_{mk}$-$M^y_{mk}$, $K^x_{mk}$-$K^y_{mk}$ and $A^x_{mk}$-$A^y_{mk}$ be 1D mass, stiffness and advection matrices, respectively\footnote{The supra-index ($x$ or $y$) indicates the variable that has been integrated, while the sub-indexes $m$ and $k$ indicates test and trial orders of approximation, respectively.}. 
In formulation~\eqref{testADIv1} we enrich the test space in the $x$-variable only, by increasing the polynomial order or decreasing the continuity in that direction. We get a system of the form:
\begin{equation}\label{vRMsys}
\begin{bmatrix}
\mathbf{G}_x&-\mathbf{B}_x\\
\mathbf{B}_x^{T}&0\\
\end{bmatrix}
\begin{bmatrix}
\mathbf{r}\\\mathbf{v}^{n+{1\over2}}
\end{bmatrix}
=
\begin{bmatrix}
-\mathbf{L}_{x}\\0 
\end{bmatrix}.
\end{equation}}
%
%
\inblue{The $\mathbf{G}_x$ matrix corresponds to the Gram matrix associated with the inner product in $\mathbf{V}_{\boldsymbol i,0}$, which only consider derivatives in the enriched $x$-direction, i.e., 
$$
\mathbf{G}_x =\begin{bmatrix}M_{mm}^{x}+K_{mm}^{x}&0\\0&M_{mm}^{x}+K_{mm}^{x}\end{bmatrix}\otimes M_{kk}^{y}.
$$
The $\mathbf{B}_x$ matrix corresponds to the left hand side of~\eqref{testADIv1}, which mixes the integration of trial and test basis functions, i.e.,  
$$
\mathbf{B}_x =
\begin{bmatrix}M^{x}_{mk}+\frac{\tau}{2}K^{x}_{mk}&0\\0&M^{x}_{mk}+\frac{\tau}{2}K^{x}_{mk}\end{bmatrix}\otimes M^{y}_{kk}
$$}
\begin{remark}
\inblue{Notice that $M^y_{kk}$ is symmetric, which combined with the Kronecker product
property $(C\otimes M)^{T}=C^{T}\otimes M^{T}$, it implies that the whole matrix on the left-hand side of~\eqref{vRMsys} can be factorized as a matrix of integrations in the $x$-variable, times $M^y_{kk}$. This is the essential property that allows us to obtain a linear cost solver.}
\end{remark}
The vector $\mathbf{L}_x$ on the right-hand side of equation~\eqref{vRMsys} corresponds to the right-hand side of the variational formulation~\eqref{testADIv1}. It has the form:%
\begin{equation}\label{eq:L_x}
\begin{array}{rl}
\mathbf{L}_x = & 
 \begin{bmatrix}M^{x}_{mk}\otimes(M^{y}_{kk}-\frac{\tau}{2}K^{y}_{kk}) &0\\0&M^{x}_{mk}\otimes(M^{y}_{kk}-\frac{\tau}{2}K^{y}_{kk})\end{bmatrix}\begin{bmatrix}v_{1}^{n}\\v_{2}^{n}\end{bmatrix}\\\\
&
  \inblue{-\begin{bmatrix}
  (v_1^n)^T N^x v_1^n + (v_2^n)^T N^y v_1^n \\
  (v_1^n)^T N^x v_2^n + (v_2^n)^T N^y v_2^n
  \end{bmatrix}}\\\\ &
  -\displaystyle\frac{\tau}{2}
\begin{bmatrix}A^{x}_{mk}\otimes M^{y}_{kk}&0\\0&M^{x}_{mk}\otimes A^{y}_{kk}\end{bmatrix}\begin{bmatrix}\tilde{p}^{n+\frac{1}{2}}\\\tilde{p}^{n+\frac{1}{2}}\end{bmatrix}
+\frac{\tau}{2}\begin{bmatrix}F_{1}^{n+\frac{1}{2}}\\F_{2}^{n+\frac{1}{2}}\end{bmatrix}.
\end{array}
\end{equation}
\inblue{where $N^x$ and $N^y$ are 3rd-order tensors with coefficients 
$$
\begin{array}{rl}
N^x_{(ij)(pq)(rs)} & =\displaystyle \int_{\Omega} \tilde B_i(x)\tilde B_j(y)\, B_p(x) B_q(y)\,  \partial_x\tilde B_r(x)\tilde B_s(y)\,dxdy,\\\\
N^y_{(ij)(pq)(rs)} & =\displaystyle \int_{\Omega} \tilde B_i(x)\tilde B_j(y)\, B_p(x) B_q(y)\,  \tilde B_r(x)\partial_y\tilde B_s(y)\,dxdy.
 \end{array}
$$
The tilde notation $\tilde B$ indicates one dimensional B-spline related to trial basis functions (otherwise is a one-dimensional test B-spline). Observe that these coefficients respect a Kronecker product structure, which allows a quick calculation of them.
}

\begin{remark}
\inblue{Notice that the second-order derivative in the right hand side of~\eqref{testADIv1} has been treated \emph{weakly} in~\eqref{eq:L_x}, by means of integration by parts. Even though this is not allowed at the continuous level, it is not a problem at the discrete level since we are using $\HH$-conforming basis functions.}
\end{remark}

\revision{The discrete treatment of the variational formulations~\eqref{testADIv2} 
will be completely analogous, taking care of the direction in which we enrich the test space.}	
For~\eqref{testADIv2} we get the system:
\begin{equation}
\begin{bmatrix}
\mathbf{G}_y&-\mathbf{B}_y\\
\mathbf{B}_y^{T}&0\\
\end{bmatrix}
\begin{bmatrix}
\mathbf{r}\\\mathbf{v}^{n+1}
\end{bmatrix}
=
\begin{bmatrix}
-\mathbf{L}_{y}\\0 
\end{bmatrix}.
\end{equation}
The $\mathbf{G}_y$ matrix corresponds to the Gram matrix associated with the inner product in $\mathbf{V}_{\boldsymbol j,0}$, which only consider derivatives in the enriched $y$-direction, i.e., 
$$
\mathbf{G}_y=M_{kk}^{x}\otimes
\begin{bmatrix}M_{mm}^{y}+K_{mm}^{y} &0\\0&M_{mm}^{y}+K_{mm}^{y}\end{bmatrix}.
$$
The $\mathbf{B}_y$ matrix corresponds to the left hand side of~\eqref{testADIv2}, which mixes trial and test basis functions, i.e.,  
$$
\mathbf{B}_y = M^{x}_{kk}\otimes
\begin{bmatrix}M^{y}_{mk}+\frac{\tau}{2}K^{y}_{mk} &0\\0&M^{y}_{mk}+\frac{\tau}{2}K^{y}_{mk}\end{bmatrix}.
$$
The vector $\mathbf{L}_y$ is given by:
$$
\begin{array}{rl}
\mathbf{L}_y =  
&\begin{bmatrix}(M^{x}_{kk}-\frac{\tau}{2}K^{x}_{kk})\otimes M^{y}_{mk}&0\\0&(M^{x}_{kk}-\frac{\tau}{2}K^{x}_{kk})\otimes M^{y}_{mk}\end{bmatrix}\begin{bmatrix}v_{1}^{n+\frac{1}{2}}\\v_{2}^{n+\frac{1}{2}}\end{bmatrix}\\\\
&\displaystyle-\frac{\tau}{2}\begin{bmatrix}A^{x}_{kk}\otimes M^{y}_{mk}&0\\0&M^{x}_{kk}\otimes A^{y}_{mk}\end{bmatrix}\begin{bmatrix}\tilde{p}^{n+\frac{1}{2}}\\\tilde{p}^{n+\frac{1}{2}}\end{bmatrix}+\frac{\tau}{2}\begin{bmatrix}F_{1}^{n+\frac{1}{2}}\\F_{2}^{n+\frac{1}{2}}\end{bmatrix}.
\end{array}
$$
\revision{The formulations~\eqref{testADIpsi} and ~\eqref{testADIphi} are stable and they do not require the iGRM stabilization.}

\section{Numerical results}
In this section, we provide several numerical results for non-stationary Stokes and Navier-Stokes problems.
\revision{First, in Section 4.1, we show that the Navier-Stokes simulation using the alternating directions implicit solver without the RM stabilization produces some oscillations for high Reynolds number.
In our paper, we use the splitting scheme from Jean-Luc Guermond and Petar Minev \cite{Minev}, which translates the Navier-Stokes problem into the velocity updates involving reaction-dominated diffusion kind of equation, followed by the explicit pressure updates.
This kind of setup is unstable for high Reynolds number in a similar manner how the reaction-dominated diffusion equations are unstable for a very small diffusion coefficient. Thus, in Section 4.1, we show how the stabilization for high Reynolds number of the Navier-Stokes formulation is performed using the residual minimization method.}
\revision{Second, in Section 4.2, we measure the error during the time-dependent simulation and test the order of the time integration scheme.}
\revision{Third, in Section 4.3, for Navier-Stokes simulations with manufactured solution, we compute $L^2$ and $H^1$ errors of the velocity and $L^2$ error of the pressure. We perform numerical experiments with different trial and test spaces, tracking their influence on velocity and pressure errors.}
\revision{Finally, in Section 4.4, we test the CFL condition for the Navier-Stokes problem with manufactured solutions, using a high Reynolds number, with the residual minimization procedure and different time-step sizes.}

\revision{The comprehensive comparison of the convergence properties, the algebraic structure, and the spectral properties of the IGA and spectral element methods for mass, stiffness, and advection-diffusion matrices is presented in \cite{r3}. The authors plot $h$ and $p$ convergence as measured in the $L_2$ and $H_1$ norms, for $p=1,...,24$ or $p=1,...,8$ depending on the formulation. They also plot the matrix sparsity patterns and analyze the spectral properties of the matrices. When comparing our paper with \cite{r3}, we have noticed that in our paper, we test the $h$ convergence in $L_2$ and $H_1$ norms for different trial/test spaces for the non-stationary Navier-Stokes problem. Such the non-stationary experiments are not considered in \cite{r3}. With respect to the $p$-convergence, we consider different trial/test spaces to fit into the residual minimization setup.}

\subsection{\revision{Cavity flow problem for Navier-Stokes equations}}

\revision{We consider the non-stationary cavity flow problem over a square domain $ (0,1)^2$
\begin{equation}
\left\{
\begin{aligned}
\partial_t \V  +\inblue{(\V \cdot \grad) \V}-\frac{1}{R_e}\Delta\V + \nabla p &= 0, \\
  \Div \V      &= 0,       
\end{aligned} 
\right. \nonumber 
\end{equation}
together with vanishing boundary conditions for $\V=(v_1,v_2)$ except on the top boundary where we make $v_1(x,1)=1$.}

\revision{First, we focus on the Galerkin formulation. We select trial and test $U_h = V_h = \left(\Spl^{3,2}_0\right)^2 \times \SplO^{3,2}$ spaces. 
We employ the alternating direction implicit solver without the residual minimization. 
We run our simulations on an $80\times 80$ mesh. 
We select $R_e=1000$. We choose the time step $\tau=10^{-2}$.}


\begin{figure}[htp]
\begin{center}
\includegraphics[scale=0.25]{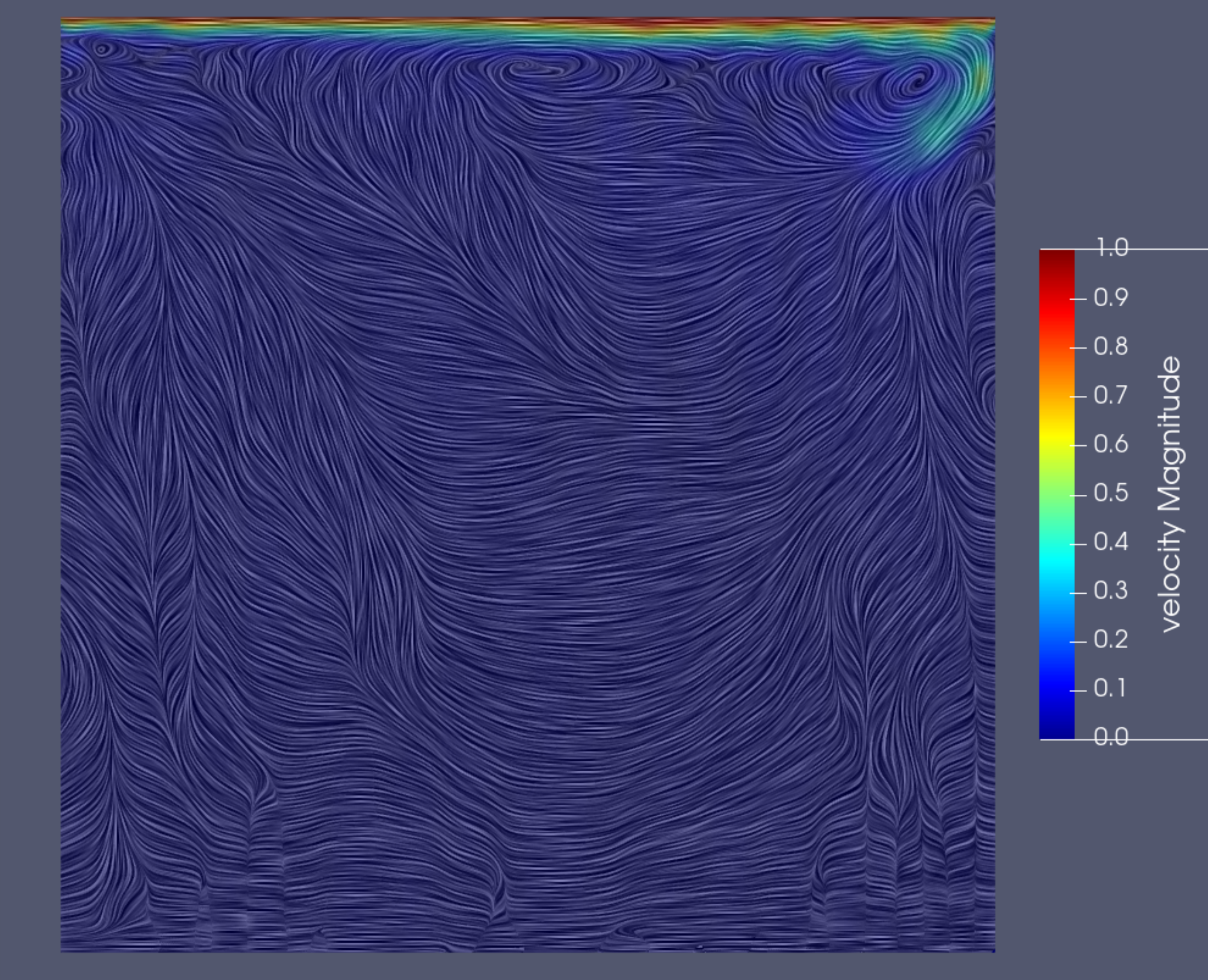} 
\includegraphics[scale=0.25]{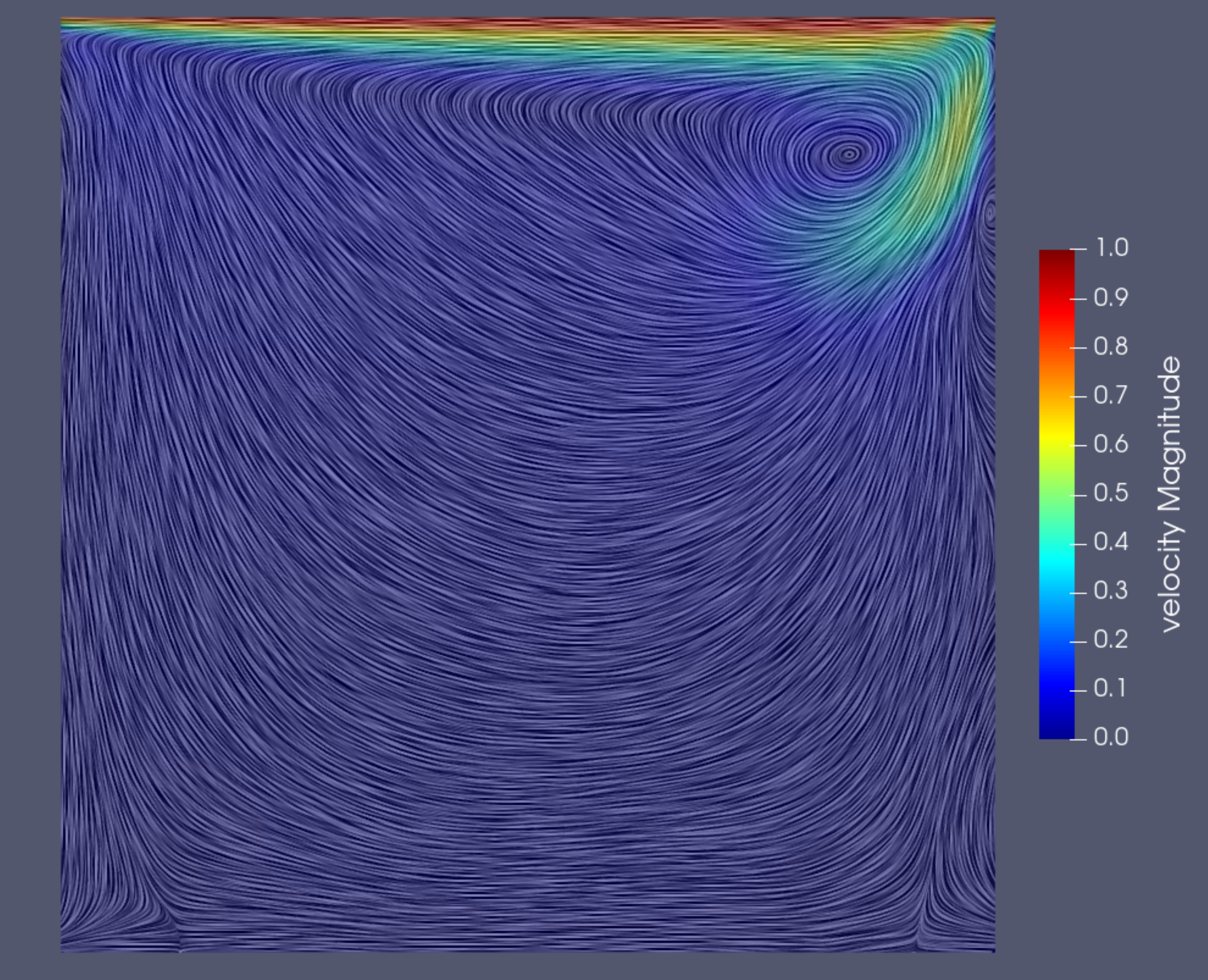} \\
\includegraphics[scale=0.25]{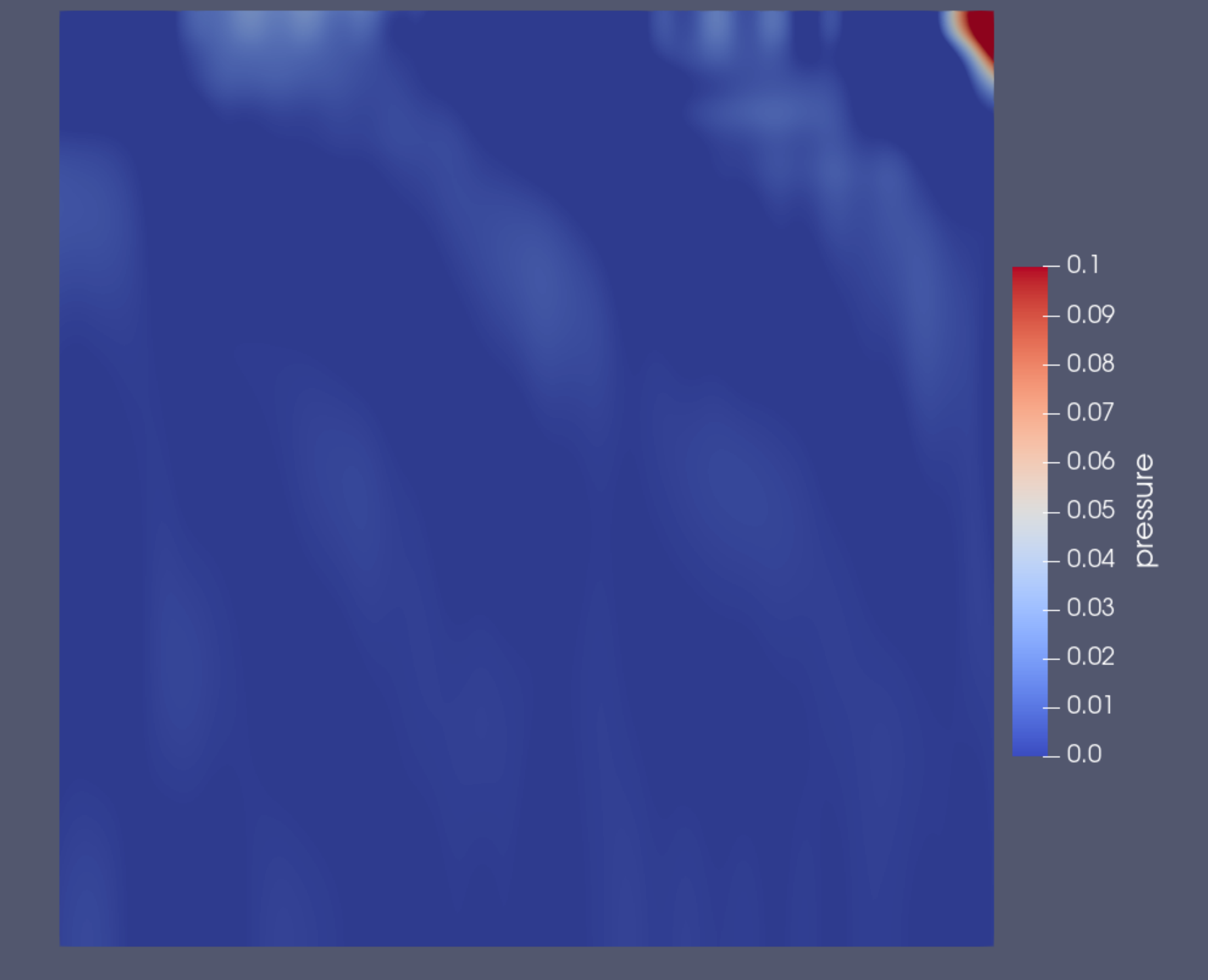} 
\includegraphics[scale=0.25]{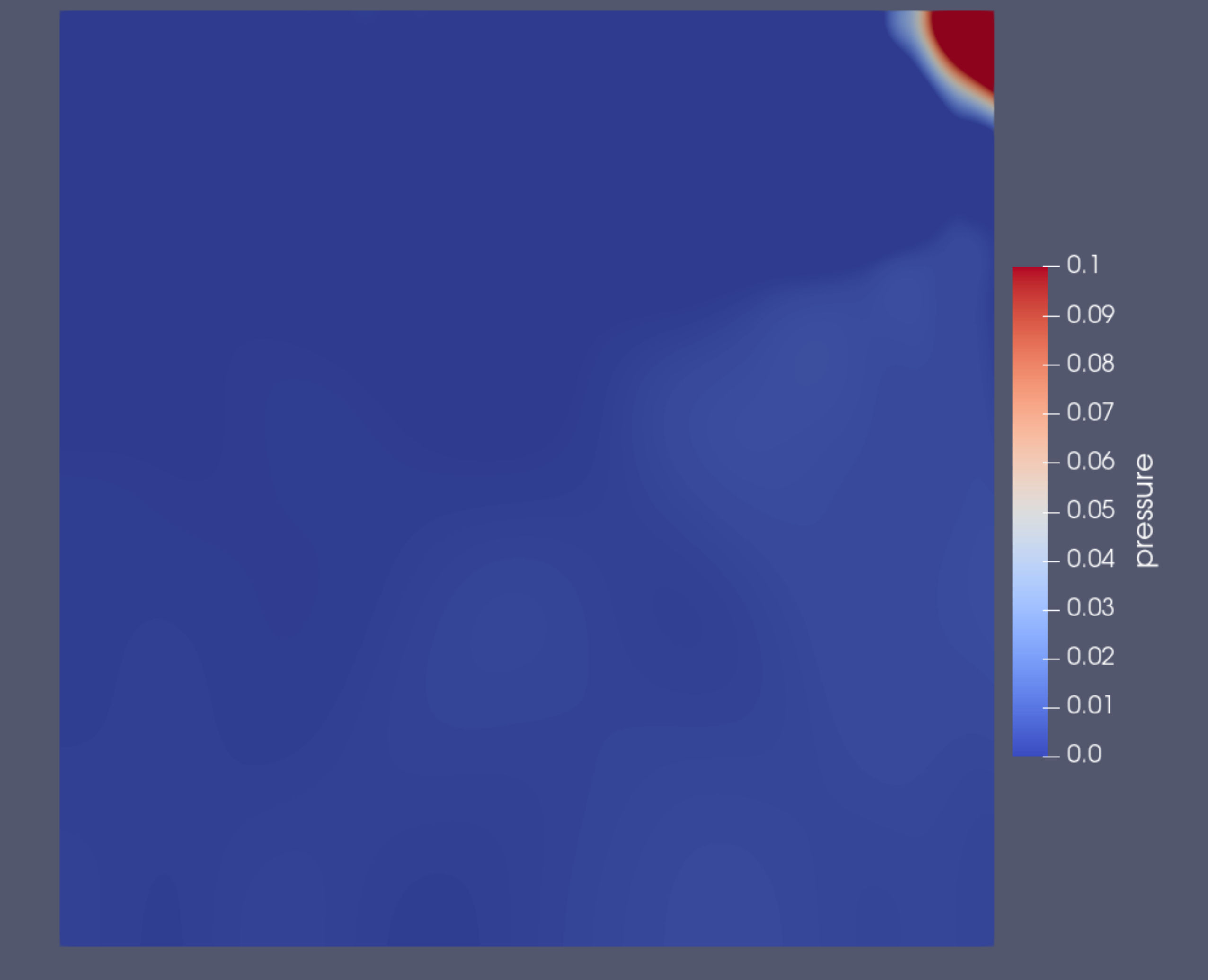} \\
\caption{\revision{Cavity flow problem for Navier-Stokes equations for $R_e=1000$. Final velocity vector field and pressure scalar field. {\bf Left column:} Alternating direction implicit solver without residual minimization for trial and test $U_h = V_h = \left(\Spl^{3,2}_0\right)^2 \times \SplO^{3,2}$ space, with a $80\times 80$ mesh. {\bf Right column:}
Alternating direction implicit solver with residual minimization for trial $U_h = \left(\Spl^{3,2}_0\right)^2 \times \SplO^{3,2}$ and test $V_h = \left(\Spl^{4,2}_0\right)^2 \times \SplO^{4,2}$ spaces, with a $80\times 80$ mesh.}}
\label{fig:unstable}
\end{center}
\end{figure}

\revision{Second, we focus on the isogeometric residual minimization method.
We select trial $U_h = \left(\Spl^{3,2}_0\right)^2 \times \SplO^{3,2}$ and test $V_h = \left(\Spl^{4,2}_0\right)^2 \times \SplO^{4,2}$ spaces.
We employ the alternating direction implicit solver, this time with the residual minimization.
We run our simulations on an $80\times 80$ mesh. We choose the time step $\tau=10^{-2}$.
We select $R_e=1000$.}

\revision{The comparison of the final configurations of the velocity and pressure are presented in Figure \ref{fig:unstable}. We can see that simulations without the residual minimization lead to some unexpected oscillations. For $R_e=1000$, where the Galerkin simulation was unstable, the residual minimization method stabilizes the problem; see Figure \ref{fig:unstable}. When running both simulations for $R_e=100$, the residual minimization method provides similar results as the Galerkin method, as presented in Figure \ref{fig:NS}. }

\revision{We also illustrate in Figures \ref{fig:NSvelL2}-\ref{fig:NSvelH1} the evolution of the $L^2$ and $H^1$ norms for the velocity . Additionally, in Figure\ref{fig:NSpresL2}, we show the $L^2$ norm of the pressure. The plots correspond to the entire simulation performed for $R_e=1000$ with the residual minimization method.}
\begin{figure}[htp]
\begin{center}
\includegraphics[scale=0.27]{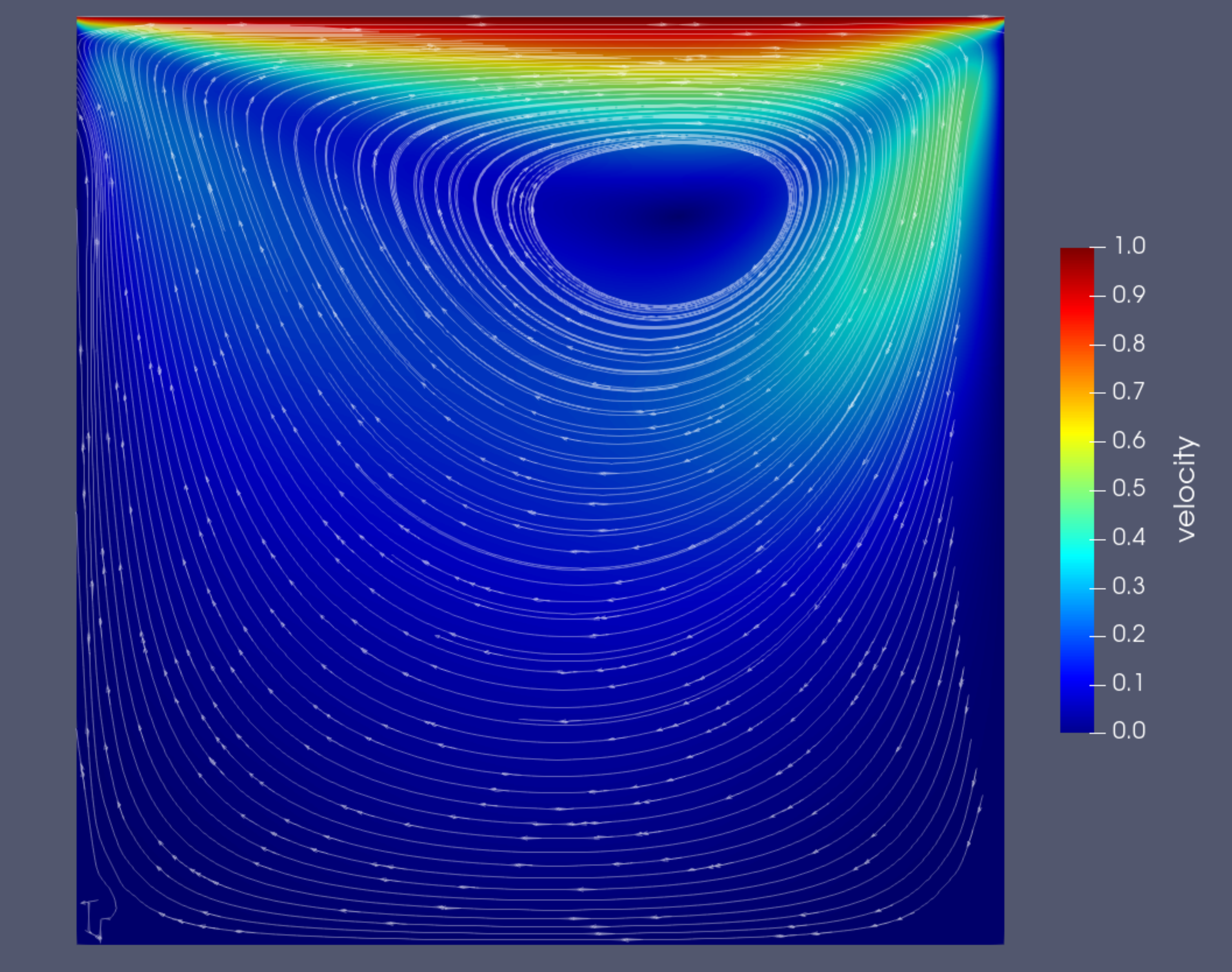} 
\includegraphics[scale=0.27]{Re-1e2-40x40.pdf} 
\caption{\revision{Cavity flow problem for Navier-Stokes equations for $R_e=100$. {\bf Left panel:} Galerkin with with trial $U_h = \left(\Spl^{3,2}_0\right)^2 \times \SplO^{3,2}$ space, with a $40\times 40$ mesh. {\bf Right panel:} Residual minimization method with trial $U_h = \left(\Spl^{3,2}_0\right)^2 \times \SplO^{3,2}$ and test $V_h = \left(\Spl^{4,2}_0\right)^2 \times \SplO^{4,2}$ spaces, with a $40\times 40$ mesh.}} 
\label{fig:NS}
\end{center}
\end{figure}
\begin{figure}[htp]
\begin{center}
\includegraphics[scale=1.0]{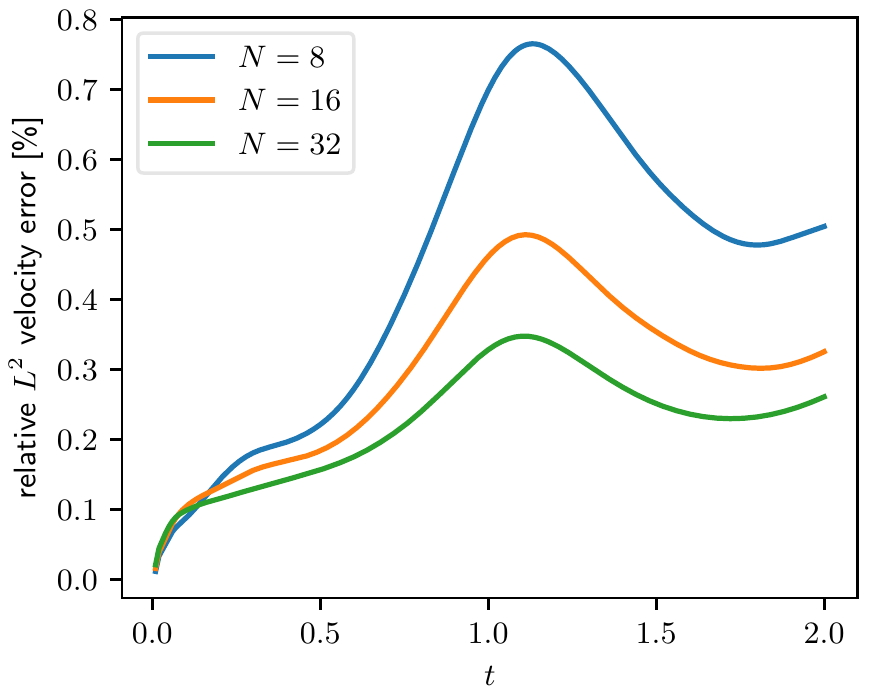} 
\caption{\revision{Cavity flow problem for Navier-Stokes equations  for $Re=1000$. $L^2$ norm of the velocity for the simulation with iGRM method with trial $U_h = \left(\Spl^{3,2}_0\right)^2 \times \SplO^{3,2}$ and test $V_h = \left(\Spl^{4,2}_0\right)^2 \times \SplO^{4,2}$ spaces, for grids with different dimensions, for $Re=1000$.}} 
\label{fig:NSvelL2}
\end{center}
\end{figure}
\begin{figure}[htp]
\begin{center}
\includegraphics[scale=1.0]{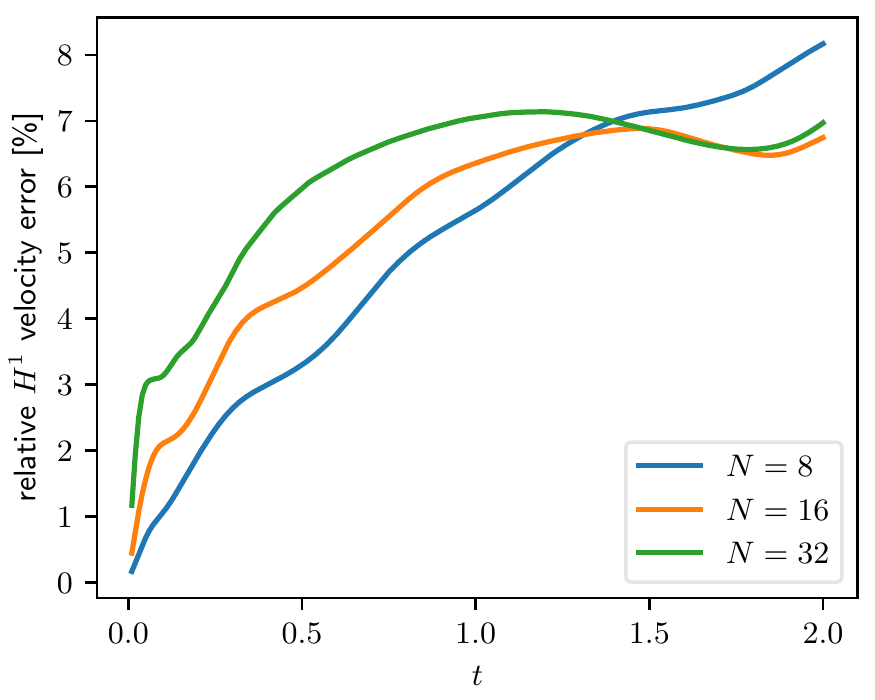} 
\caption{\revision{Cavity flow problem for Navier-Stokes equations  for $Re=1000$. $H^1$ norm of the velocity for the simulation with iGRM method with $U_h = \left(\Spl^{3,2}_0\right)^2 \times \SplO^{3,2}$ and test $V_h = \left(\Spl^{4,2}_0\right)^2 \times \SplO^{4,2}$ spaces, for grids with different dimensions, for $Re=1000$.}}
\label{fig:NSvelH1}
\end{center}
\end{figure}
\begin{figure}[htp]
\begin{center}
\includegraphics[scale=1.0]{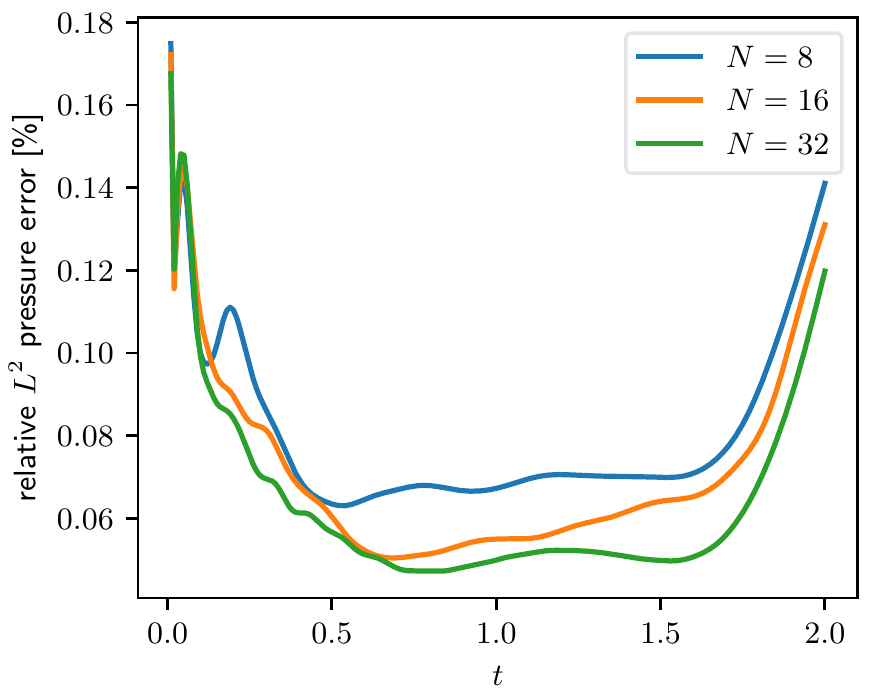} 
\caption{\revision{Cavity flow problem for Navier-Stokes equations for $Re=1000$. $L^2$ norm of the pressure for the simulation with iGRM method with trial $U_h = \left(\Spl^{3,2}_0\right)^2 \times \SplO^{3,2}$ and test $V_h = \left(\Spl^{4,2}_0\right)^2 \times \SplO^{4,2}$ spaces, for grids with different dimensions.}} 
\label{fig:NSpresL2}
\end{center}
\end{figure}

\subsection{The non-stationary Stokes problem with manufactured solution}
We consider the non-stationary Stokes problem in the square~$\Omega = (0, 1)^2$ and time interval $I=(0,2]$, with no-slip boundary conditions, i.e., %
$$
\left\{
\begin{array}{r}
\text{Find~$(\V,p)$ such that}\\
\partial_t \V  -\Delta\V + \nabla p = \F, \\
  \Div \V      = 0,\\
  \left.\V\right|_{\partial \Omega} = 0,\\
  \mathbf{v}(0)=\mathbf{v}_{0},\\
\end{array}
\right.
$$
with $\F$ and $\mathbf{v}_{0}$ defined in such a way that the manufactured solution is $\mathbf{v}(x,y,t) = (\sin(x)\sin(y+t), \cos(x) \cos(y+t))$ and $p(x,y) = \cos(x) \sin(y+t)$. 

\revision{{We select $U_h = \left(\Spl^{3,2}_0\right)^2 \times \SplO^{3,2}$ trial spaces and 
$V_h = \left(\Spl^{4,2}_0\right)^2 \times \SplO^{4,2}$ test spaces}, over a mesh of $40\times 40$ elements. We employ the linear  {${\cal O}{(N)}$}  computational cost solver using the Kronecker product structure at each sub-step of the time iteration algorithm.}
\revision{We measure the influence of different time step sizes into the numerical error of our problem.} The $L^2$-error of the solutions at final time $t=2$ is reported in Figure \ref{fig:L2H1rel}.

\begin{figure}[htp]
\begin{center}
\includegraphics[scale=0.4]{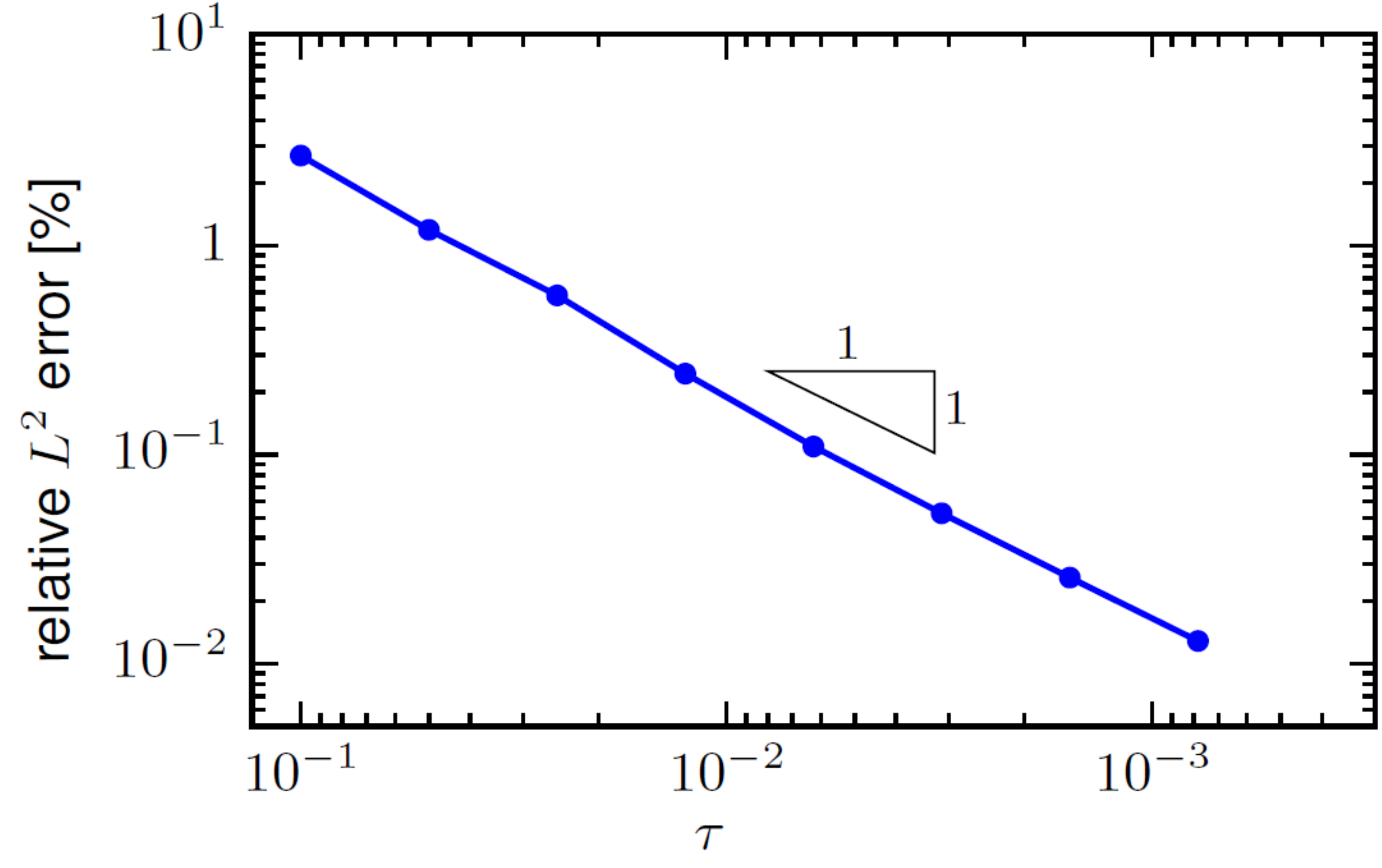} 
\caption{\revision{The non-stationary Stokes problem with manufactured solution. $L^2$ norm of the solution at final time $t=2$ for iGRM method with}
$U_h = \left(\Spl^{3,2}_0\right)^2 \times \SplO^{3,2}$ trial space and 
$V_h = \left(\Spl^{4,2}_0\right)^2 \times \SplO^{4,2}$ test space, for different time step sizes $\tau$ and a fixed mesh of $40\times 40$ elements.}
\label{fig:L2H1rel}
\end{center}
\end{figure}

\inblue{As described in \cite[Proposition 4.1]{Minev}, 
the perturbed scheme~\eqref{perturbed} will be first-order in time if the $\chi$ parameter is equal to 0. Selecting $\chi>0$ and $\phi^{-\frac{1}{2}}=\frac{\tau}{2}\partial_t p(0)$ may increase the order of the time-stepping scheme (cf.~\cite{GueMinSal}). The proposition concerns the convergence of the perturbed problem to the unperturbed Navier-Stokes equation, when $\epsilon$ is set to be equal to the time step size. In our simulations, we select $\chi=0$ and we recover a first-order accurate time integration scheme. However, when we run the experiments using  $\chi>0$ and $\phi^{-\frac{1}{2}}=\frac{\tau}{2}\partial_t p(0)$, we still obtain a first-order dependence with respect to time. }

\subsection{A Navier-Stokes problem with manufactured solution}
\inblue{We consider the following Navier-Stokes equation over the square {$\Omega = (0,1)^2$} with no-slip boundary condition %
$$
\left\{
\begin{array}{r}
\text{Find~$\V=(v_1,v_2)$ and $p$ such that:}\\
\partial_t \V    -\displaystyle{\Delta\V\over R_e} +\inblue{(\V \cdot \grad) \V} + \nabla p   = \F, \\
  \Div \V        = 0,\\
  \left.\V\right|_{\partial \Omega}   = 0,\\
  \mathbf{v}(0)  =\mathbf{v}_{0},
\end{array}
\right.
$$
with $\F$ and $\mathbf{v}_{0}$ defined in such a way that the manufactured solution is $\mathbf{v}(x,y,t) = (\sin(x) \sin(y+t), \cos(x) \cos(y+t))$ and $p(x,y,t) = \cos(x) \sin(y+t)$. }

\revision{First, we focus on the Galerkin formulation without the residual minimization method. 
We use trial and test $U_h = V_h = \left(\Spl^{3,2}_0\right)^2 \times \SplO^{3,2}$ spaces, with a $40\times 40$ mesh, $Re=1000$. We select time step size $\tau=\frac{1}{512}$ and intend to perform 1024 time steps. The simulations generates some oscilations, which is illustrated in Figure \ref{fig:unstable2}.}

\begin{figure}[htp]
\begin{center}
\includegraphics[scale=0.21]{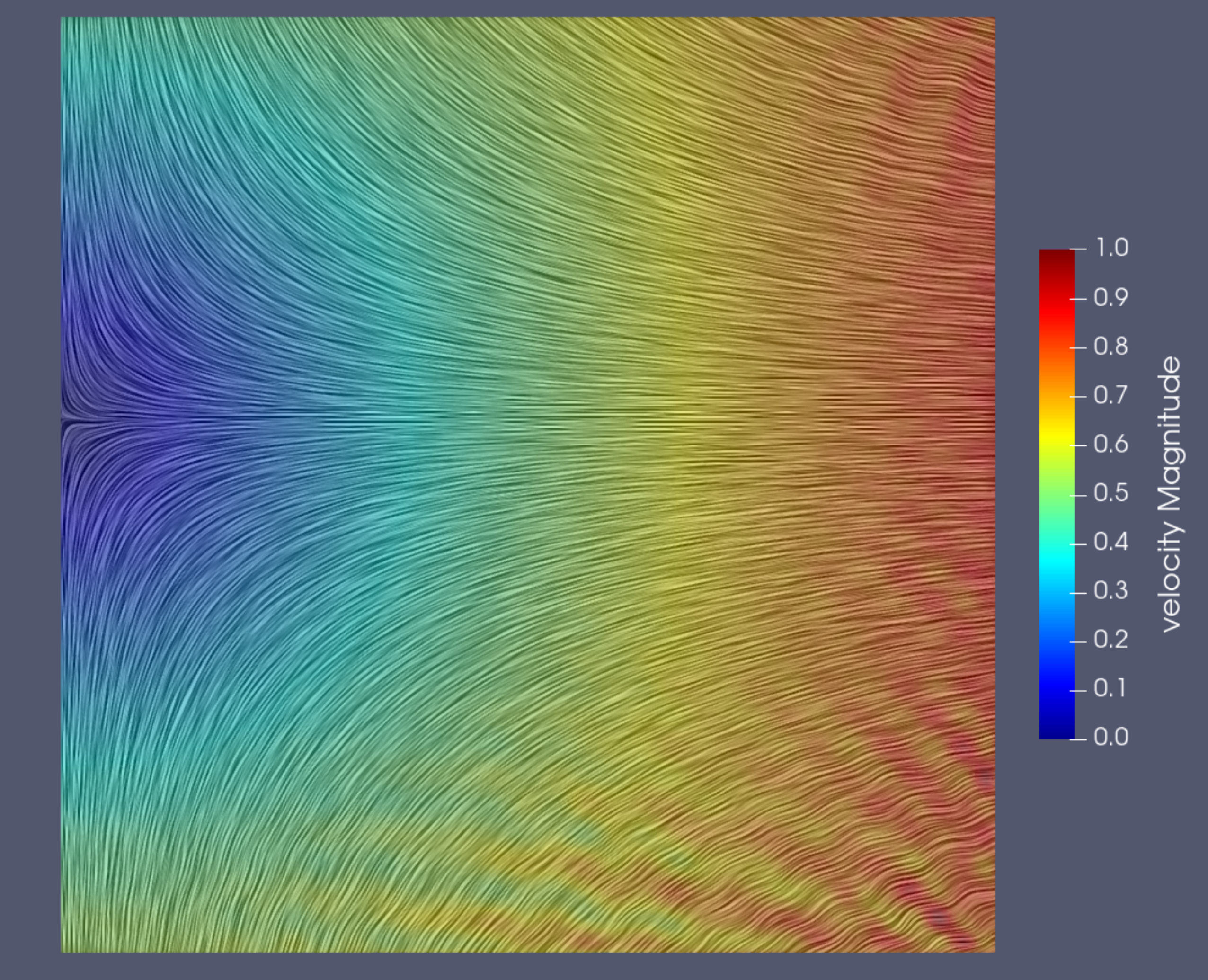} 
\includegraphics[scale=0.21]{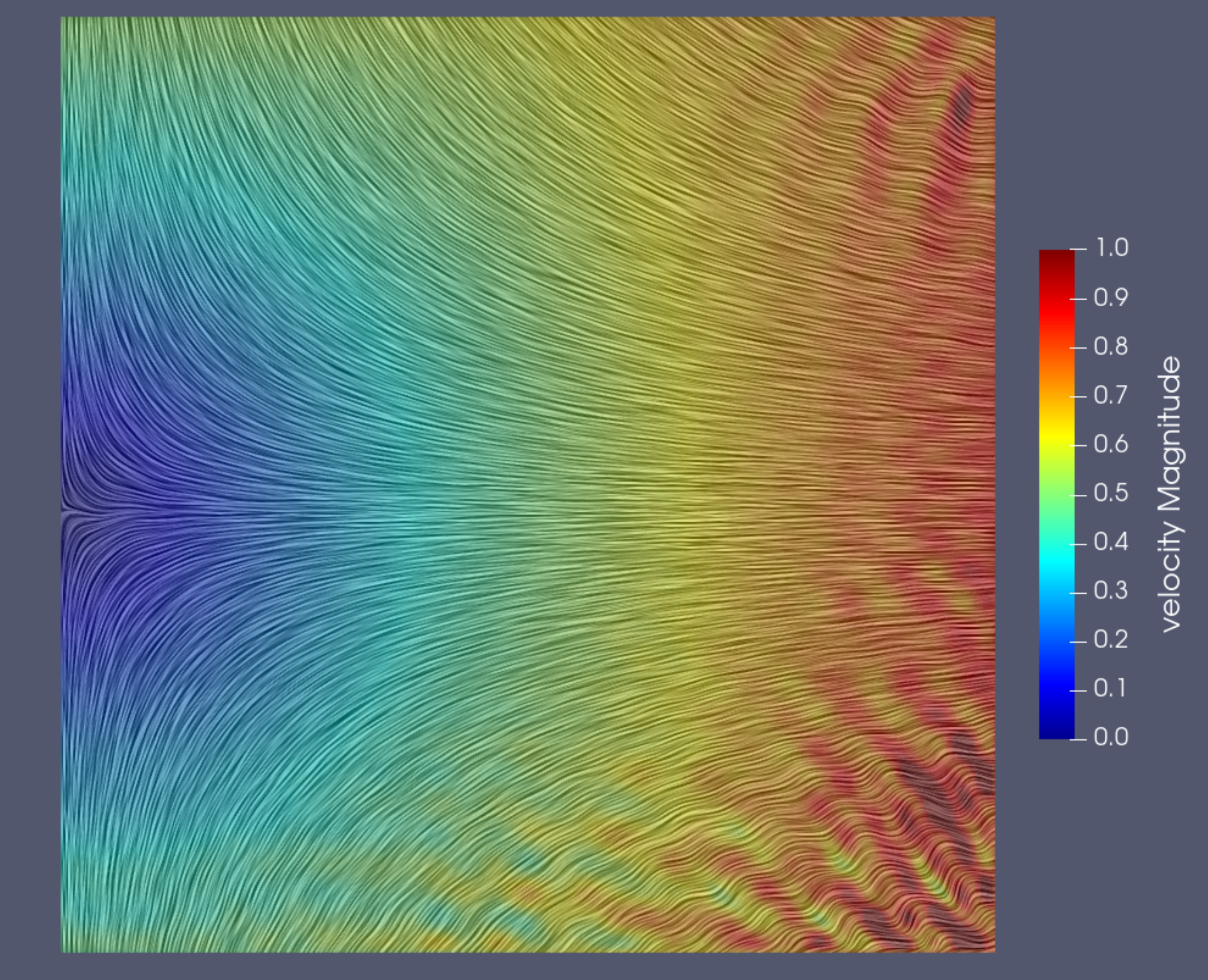} \\
\includegraphics[scale=0.21]{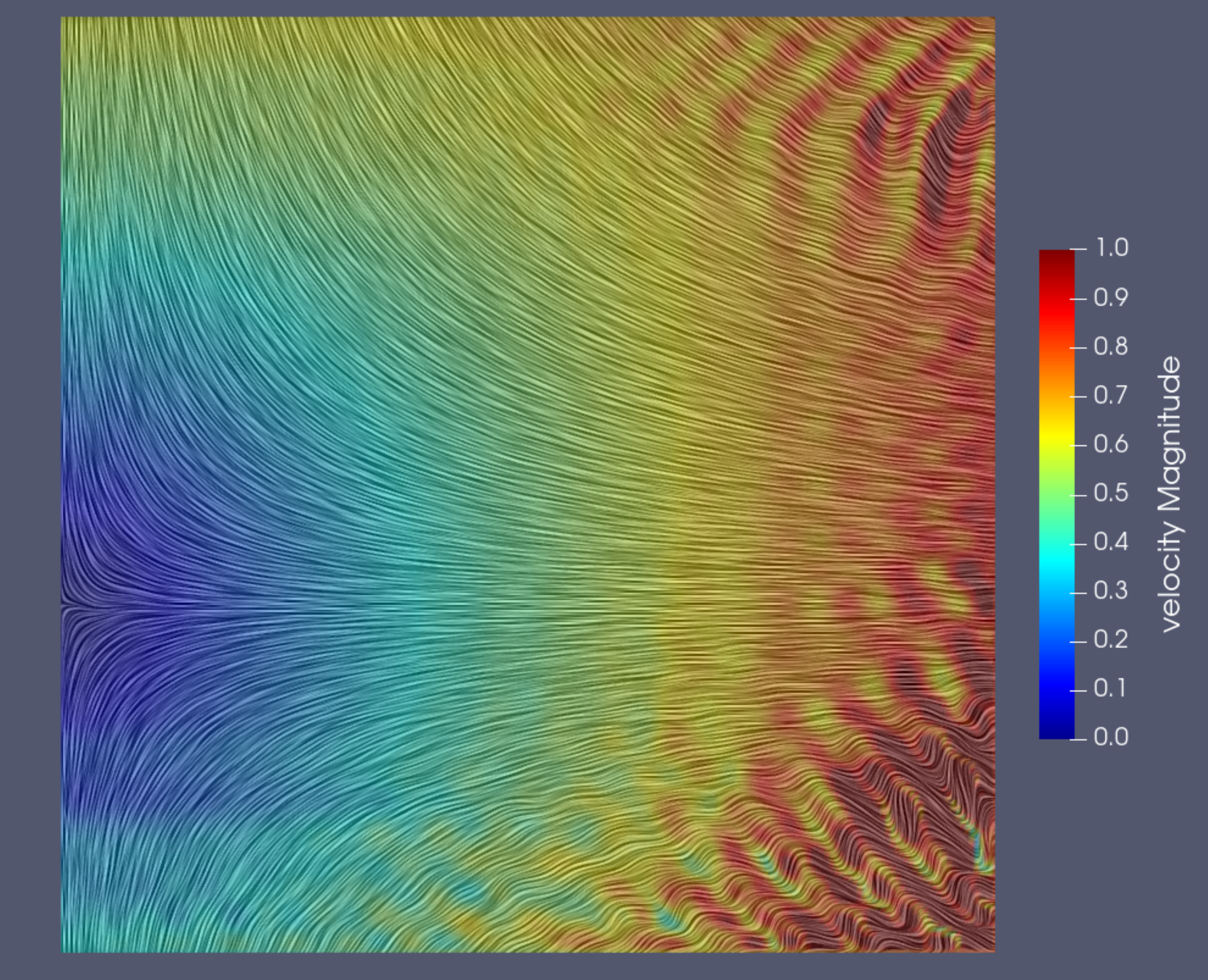} 
\includegraphics[scale=0.21]{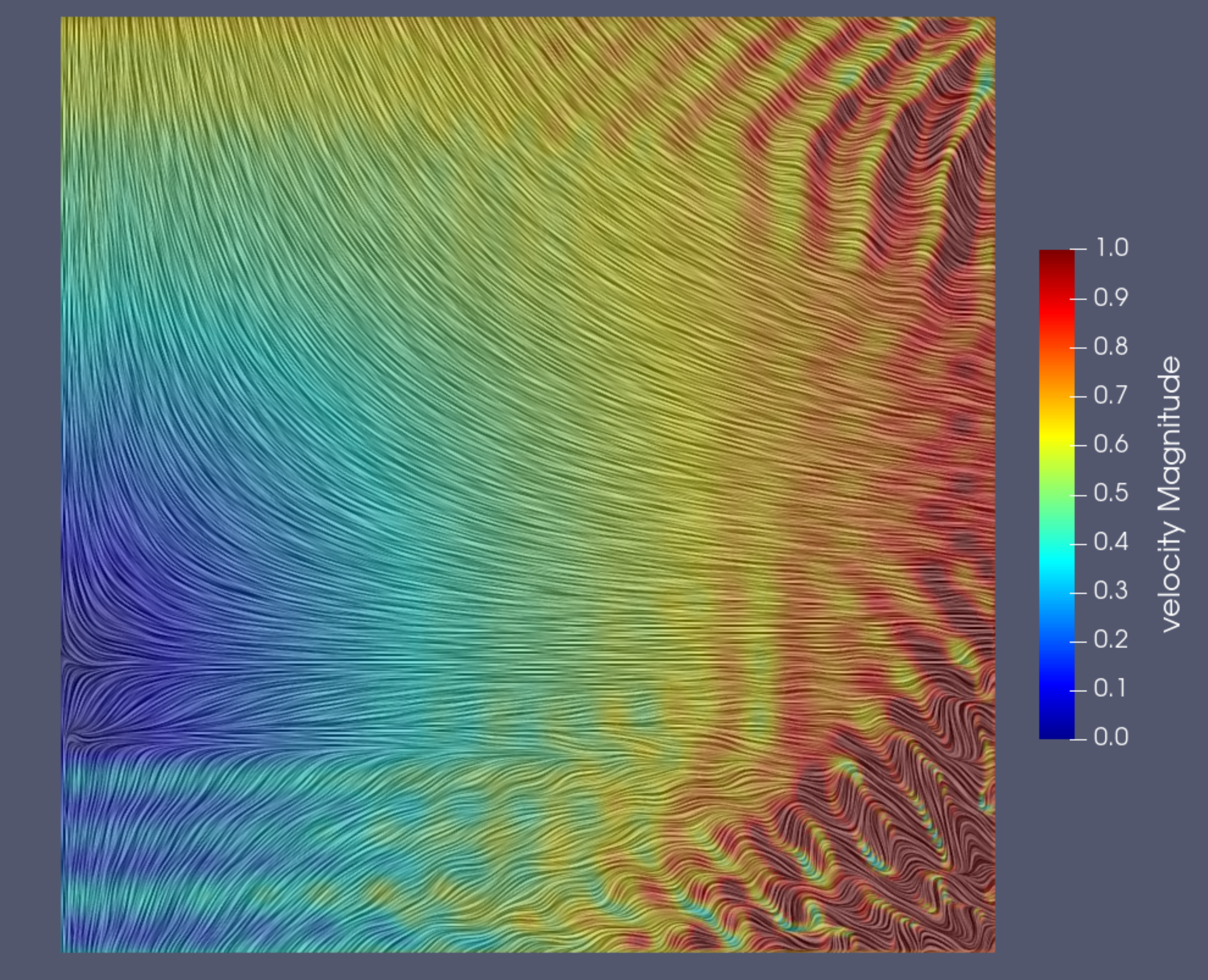} \\
\includegraphics[scale=0.21]{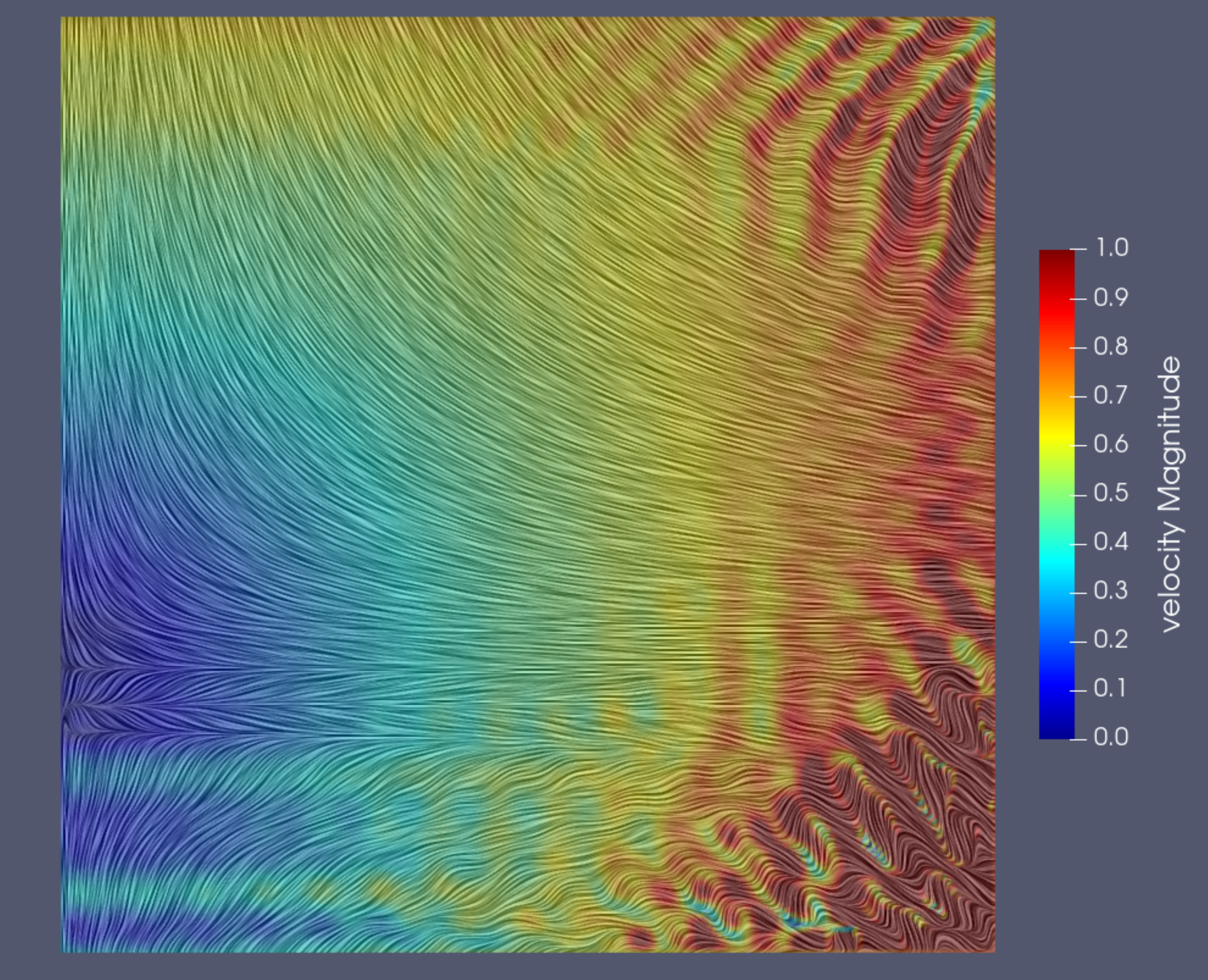} 
\includegraphics[scale=0.21]{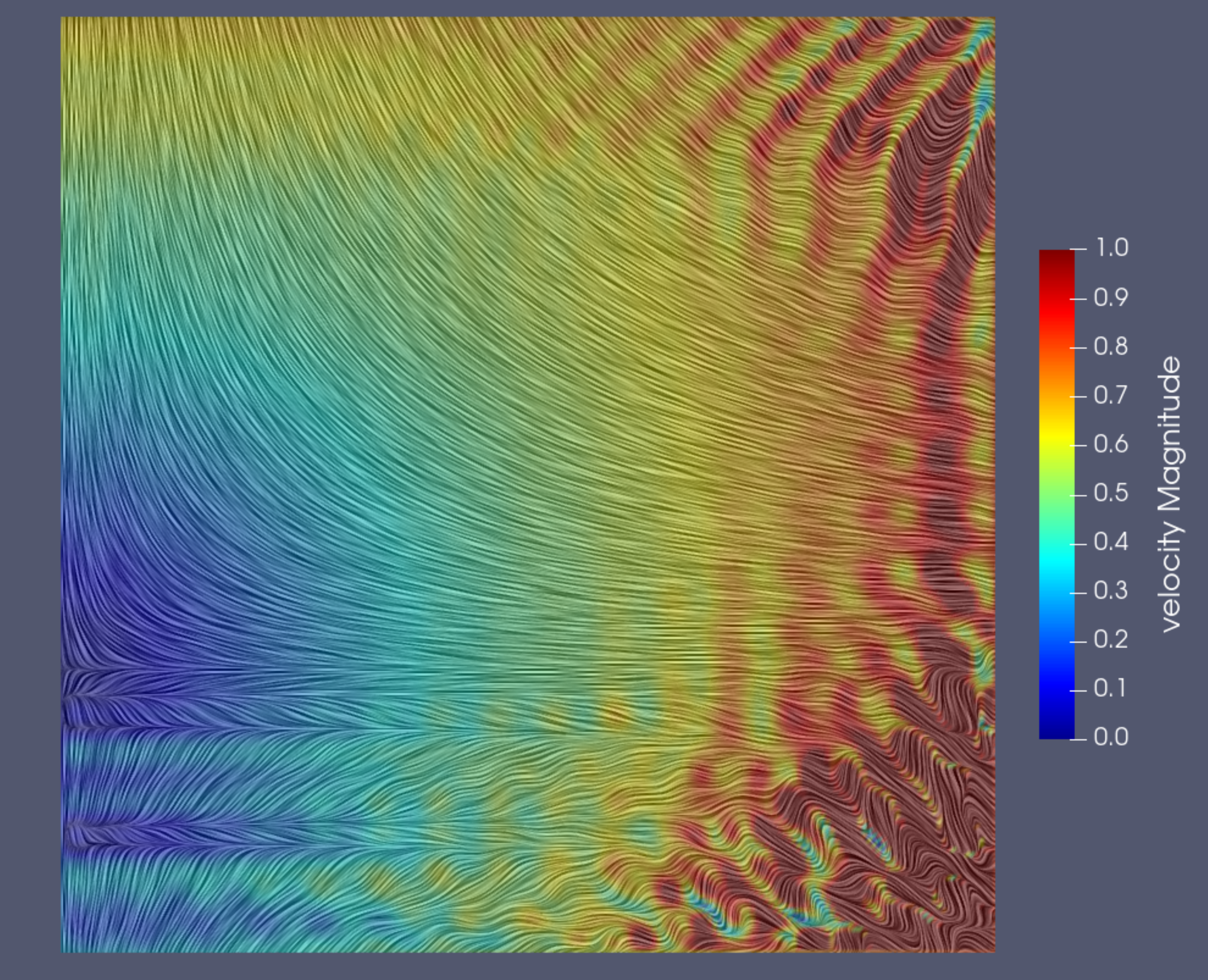} \\
\includegraphics[scale=0.21]{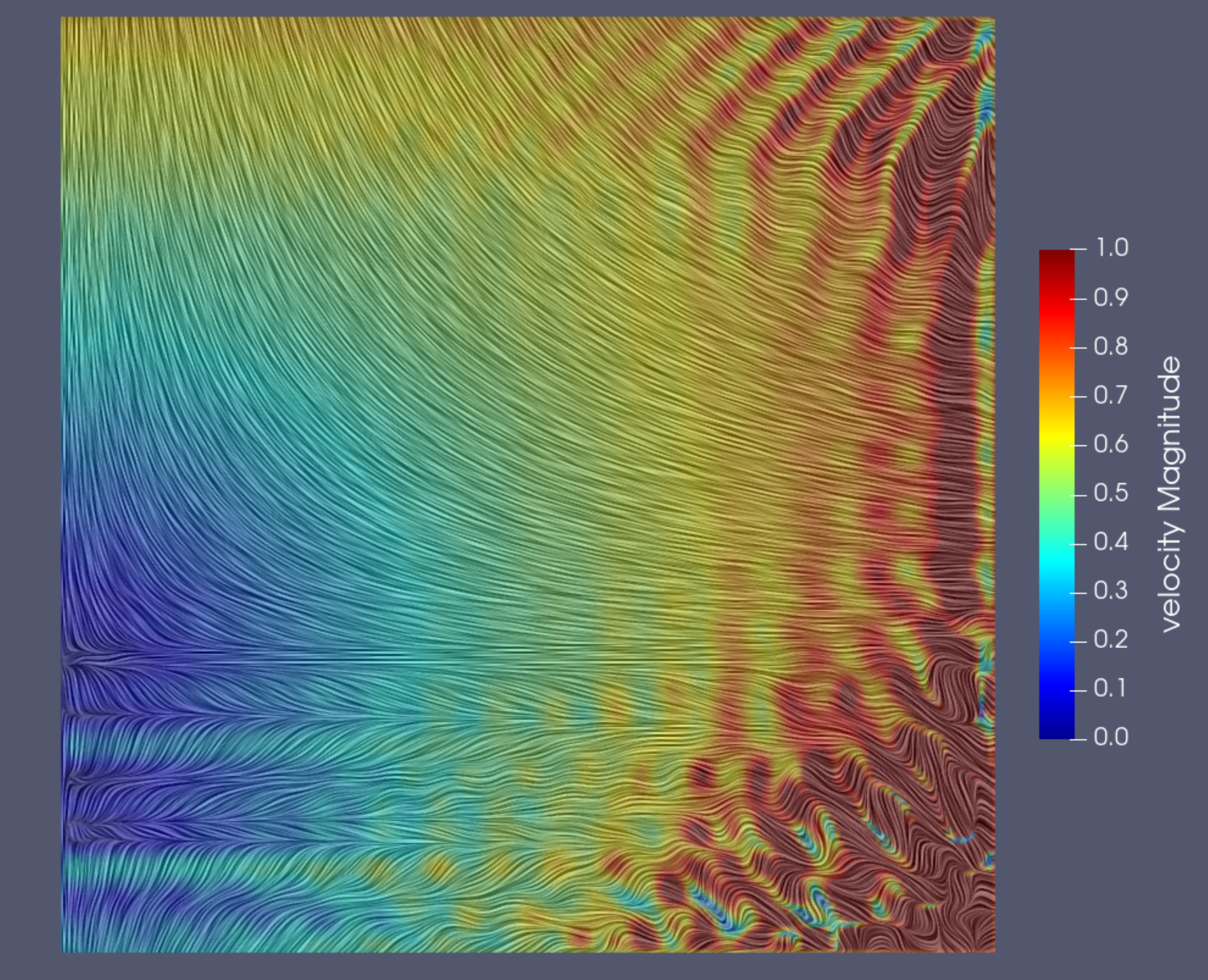} 
\includegraphics[scale=0.21]{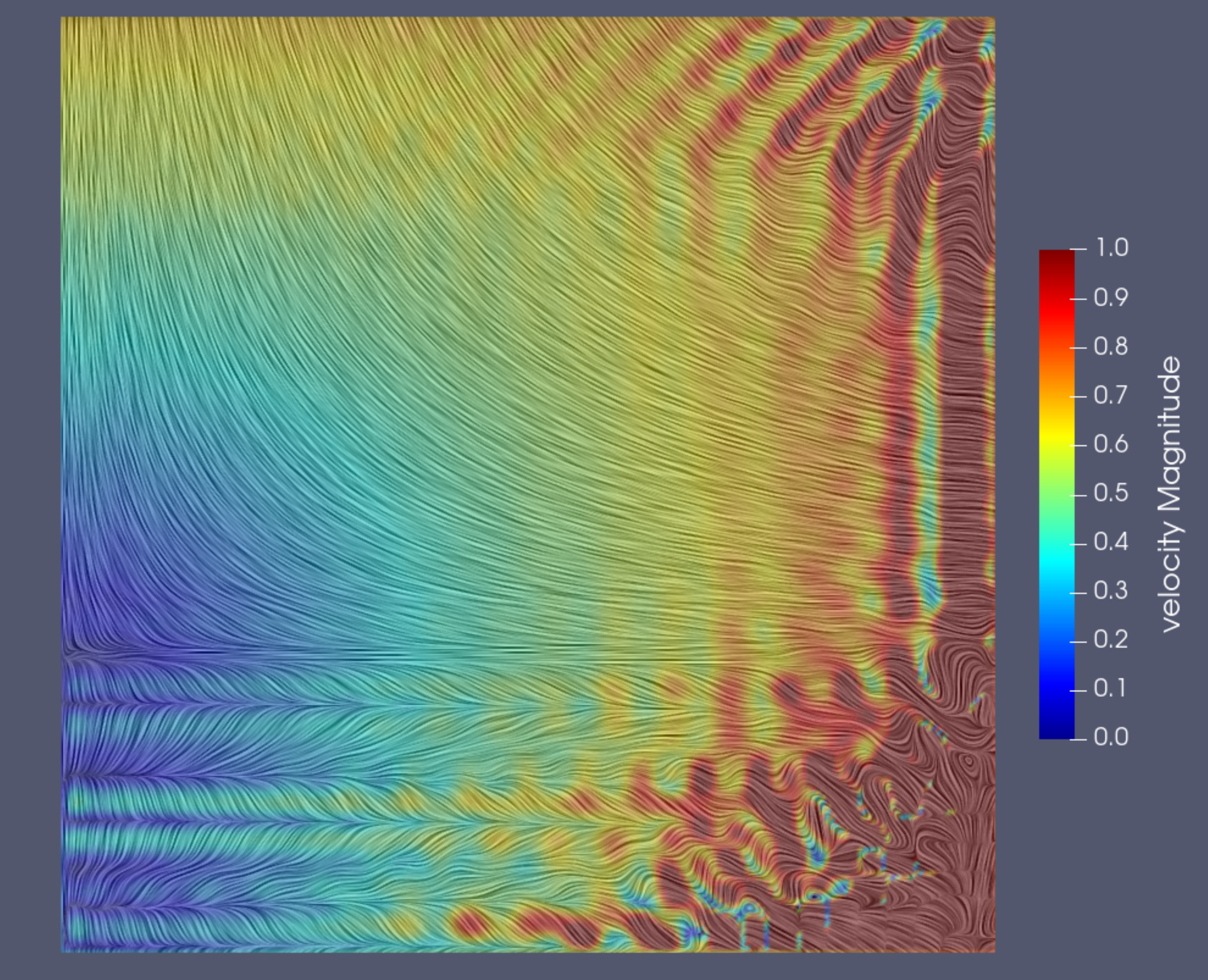} \\
\caption{\revision{A Navier-Stokes problem with manufactured solution with $Re=1000$. Alternating direction implicit solver without residual minimization for $U_h = V_h = \left(\Spl^{3,2}_0\right)^2 \times \SplO^{3,2}$ trial and test spaces, with a $40\times 40$ mesh. Snapshots on the velocity field in time moments 100, 110, 120, 125, 126, 127, 128, and 129.}} 
\label{fig:unstable2}
\end{center}
\end{figure}

\revision{Introducing the residual minimization stabilizes the simulation, and it allows us to measure the influence of different time step sizes into the numerical error for different trial and test spaces. The error measured in $L^2$ and $H^1$ norms for the velocity is reported in Figures \ref{fig:L2relvel}-\ref{fig:H1relvel}. } 
\begin{figure}[htp]
\begin{center}
\includegraphics[scale=0.4]{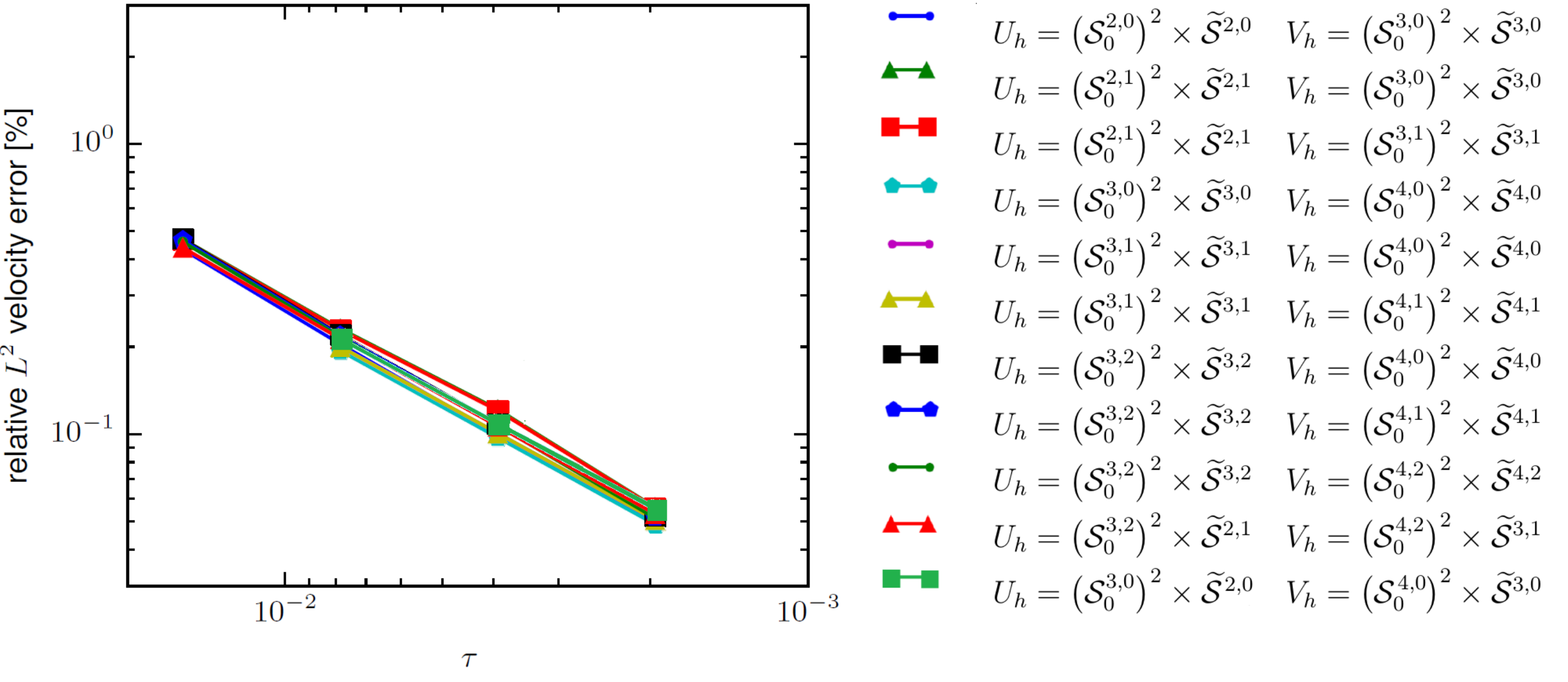} 
\caption{\revision{A Navier-Stokes problem with manufactured solution with $R_e=1000$. $L^2$ relative error of the velocity at final time $t=2$, for different trial and test spaces, for the iGRM method with a $20\times 20$ mesh.} }
\label{fig:L2relvel}
\end{center}
\end{figure}

\begin{figure}[htp]
\begin{center}
\includegraphics[scale=0.4]{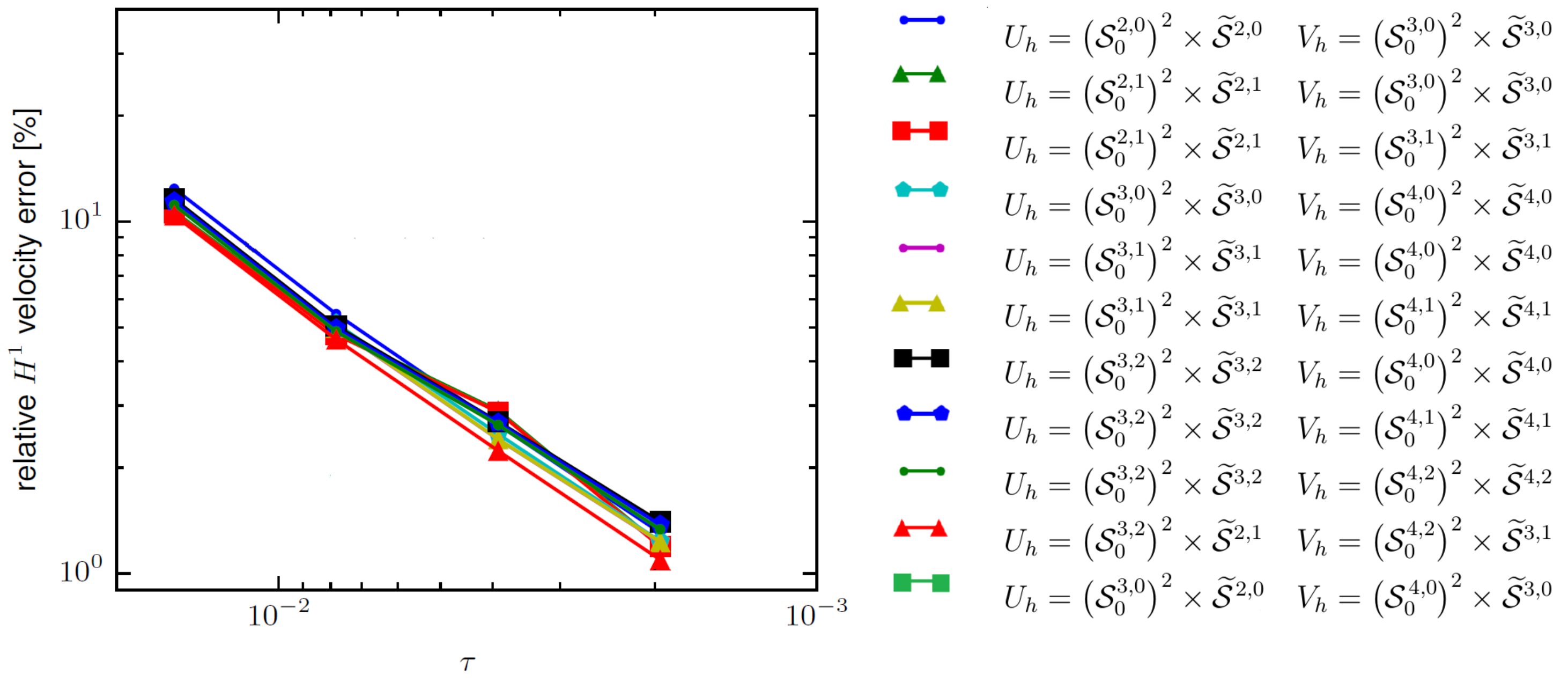} 
\caption{\revision{A Navier-Stokes problem with manufactured solution with $R_e=1000$. $H^1$ relative error of the velocity at final time $t=2$, for different trial and test spaces, for the iGRM method with a $20\times 20$ mesh.} }
\label{fig:H1relvel}
\end{center}
\end{figure}
\begin{figure}[htp]
\begin{center}
\includegraphics[scale=0.4]{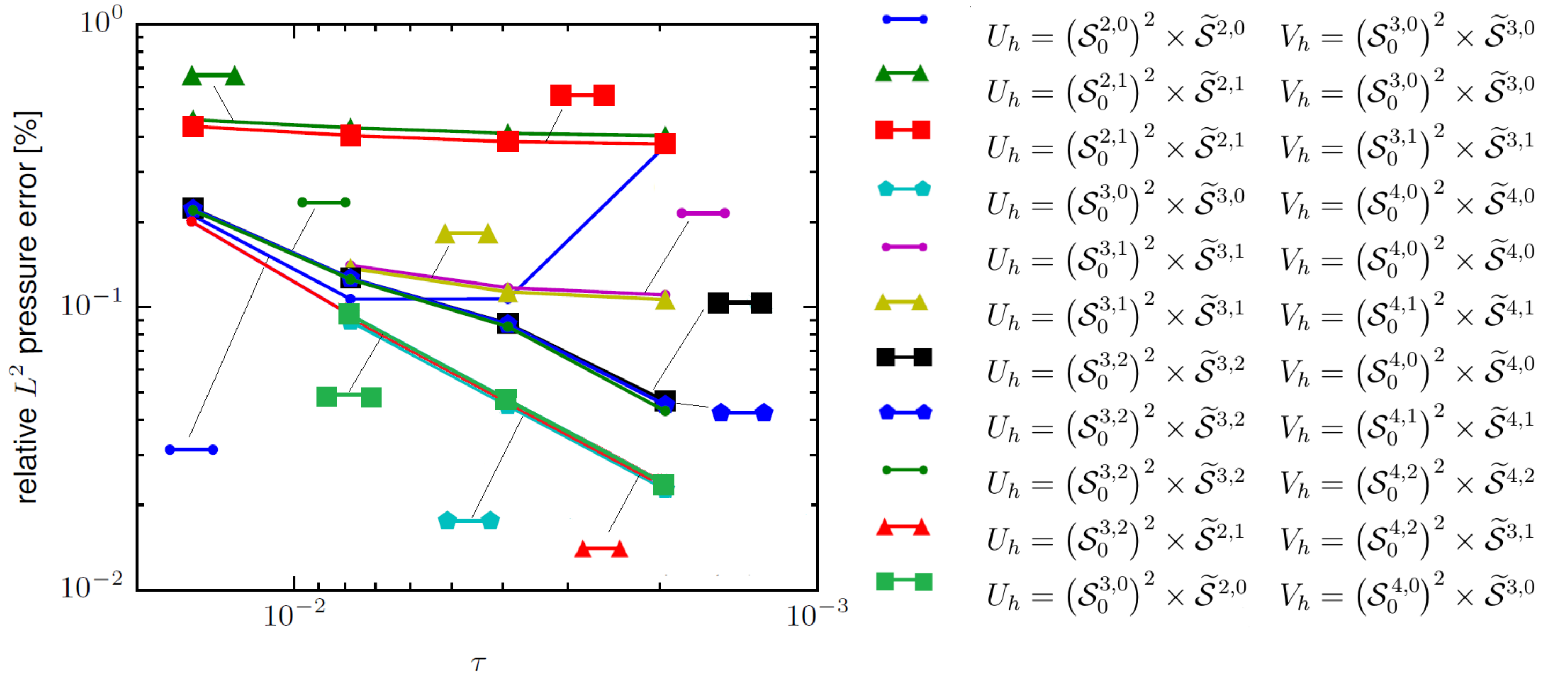} 
\caption{\revision{A Navier-Stokes problem with manufactured solution with $R_e=1000$. $L^2$ relative error of the pressure at time $t=2$, for different trial and test spaces, for the iGRM method with a $20\times 20$ mesh. }}
\label{fig:L2pres}
\end{center}
\end{figure}

\revision{We also report the $L^2$ norm error for the pressure in Figure~\ref{fig:L2pres}.
We observe different convergence rates of the pressure for different trial and test spaces. }
If we focus on the time step size $\tau=\frac{1}{512}$ and perform 1024 time steps, we obtain the final $L^2$-errors for the pressure listed in Table \ref{tab:errors}. \revision{We have denoted by bold the spaces resulting in the lowest numerical error of the pressure.}
\begin{table*}[htp]
\begin{center}
\begin{tabular}{|c|c|c|c|c|c|c|}
\hline 
Symbol & Trial $U_h$ & Test $V_h$ & $p$-error & $\dim$(trial) & $\dim$(test) & time[s] \\
\hline
\includegraphics[scale=0.3]{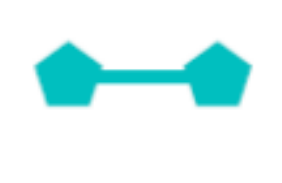} & 
$\left(\Spl^{3,0}_0\right)^2 \times \SplO^{3,0}$ & 
$\left(\Spl^{4,0}_0\right)^2 \times \SplO^{4,0}$ &  
	{\bf 0.022} & 11162 & 19683 & {\bf 2166} \\
\includegraphics[scale=0.3]{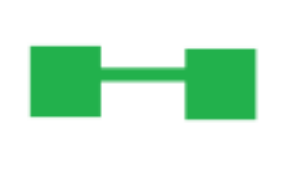} & 
$\left(\Spl^{3,0}_0\right)^2 \times \SplO^{2,0}$ &
$\left(\Spl^{4,0}_0\right)^2 \times \SplO^{3,0}$ &
{\bf 0.022} & 9123 & 16843 & {\bf 2010} \\
\includegraphics[scale=0.3]{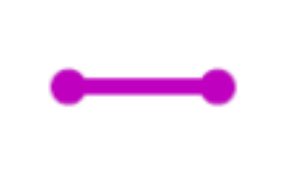} & 
$\left(\Spl^{3,1}_0\right)^2 \times \SplO^{3,1}$ &
$\left(\Spl^{4,0}_0\right)^2 \times \SplO^{4,0}$ &
0.11 & 5292 & 19683 & 1889 \\
\includegraphics[scale=0.3]{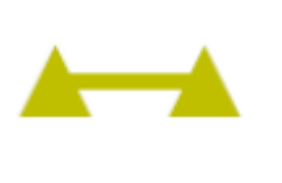} & 
$\left(\Spl^{3,1}_0\right)^2 \times \SplO^{3,1}$ &
$\left(\Spl^{4,1}_0\right)^2 \times \SplO^{4,1}$ &
0.11 & 5292 & 11532 &  1372 \\
\includegraphics[scale=0.3]{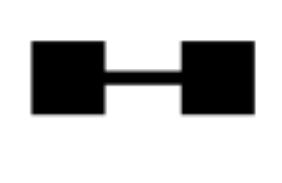} &  
$\left(\Spl^{3,2}_0\right)^2 \times \SplO^{3,2}$ &
$\left(\Spl^{4,0}_0\right)^2 \times \SplO^{4,0}$ &
0.046 & 1587 & 19683 & 1683\\
\includegraphics[scale=0.3]{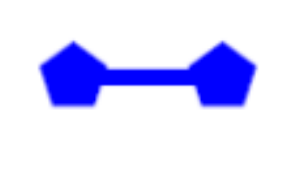} &  
$\left(\Spl^{3,2}_0\right)^2 \times \SplO^{3,2}$ &
$\left(\Spl^{4,1}_0\right)^2 \times \SplO^{4,1}$ &
0.045 & 1587 & 11532 & 854 \\
\includegraphics[scale=0.3]{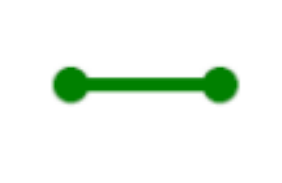} &  
$\left(\Spl^{3,2}_0\right)^2 \times \SplO^{3,2}$ &
$\left(\Spl^{4,2}_0\right)^2 \times \SplO^{4,2}$ &
0.043 & 1587 & 5547 & 329 \\ 
\includegraphics[scale=0.3]{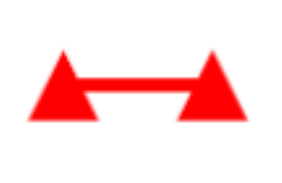} & 
$\left(\Spl^{3,2}_0\right)^2 \times \SplO^{2,1}$ &
$\left(\Spl^{4,2}_0\right)^2 \times \SplO^{3,1}$ &
{\bf 0.022} & 1542 & 5462 & {\bf 300}\\
\hline
\end{tabular}
\end{center}
\caption{\label{tab:errors} \revision{A Navier-Stokes problem with manufactured solution, with $R_e=1000$. } Relative $L^2$ pressure error, for time step size $\tau=\frac{1}{512}$, \revision{for the iGRM method }with a $20\times 20$ mesh.}
\label{tab:errors}
\end{table*}
We can also read from this table the dimensions of the trial and test spaces, as well as the total execution time for 1024 time steps. 

\revision{From Figures \ref{fig:L2relvel}-\ref{fig:L2pres} we can read that all the spaces result in the same numerical accuracy for the velocity, as measured in both $L^2$ and $H^1$ norms. We can also read that the best convergence of the $L^2$ error for pressure is obtained from 
\begin{itemize}
\item Standard finite element space trial $U_h = \left(\Spl^{3,0}_0\right)^2 \times \SplO^{3,0}$ and test
$V_h = \left(\Spl^{4,0}_0\right)^2 \times \SplO^{4,0}$ with equal-order approximation of the velocity and pressure;
\item Standard finite element space trial $U_h = \left(\Spl^{3,0}_0\right)^2 \times \SplO^{2,0}$ and test
$V_h = \left(\Spl^{4,0}_0\right)^2 \times \SplO^{3,0}$ with reduced order approximation of the pressure;
\item Higher continuity space trial $U_h = \left(\Spl^{3,2}_0\right)^2 \times \SplO^{2,1}$ and test
$V_h = \left(\Spl^{4,2}_0\right)^2 \times \SplO^{3,1}$ with reduced order approximation of the pressure.
\end{itemize}
From Table \ref{tab:errors} we can read that the higher continuity spaces result in a lower computational cost. }

\subsection{\inblue{CFL condition for Navier-Stokes simulation}}
\revision{We consider the Navier-Stokes problem with manufactured solution from section 4.3. We focus on the iGRM method to check the CFL condition for the simulation with $40\times 40$ elements and $Re=1000$. 
The simulation is performed over the time interval $[0,2]$, with time step sizes $\tau=2^{-\ell}$, for $\ell=3,...,8$. The results are summarized in Figure~\ref{fig:NScfl}. We conclude that for time steps $\tau=\frac{1}{8}$, $\tau=\frac{1}{16}$, $\tau=\frac{1}{32}$, and $\tau=\frac{1}{64}$, the simulations explodes sooner or later. The simulations with $\tau=\frac{1}{128}$ or $\tau=\frac{1}{256}$ are stable.}

\begin{figure}[htp]
\begin{center}
\includegraphics[scale=1.0]{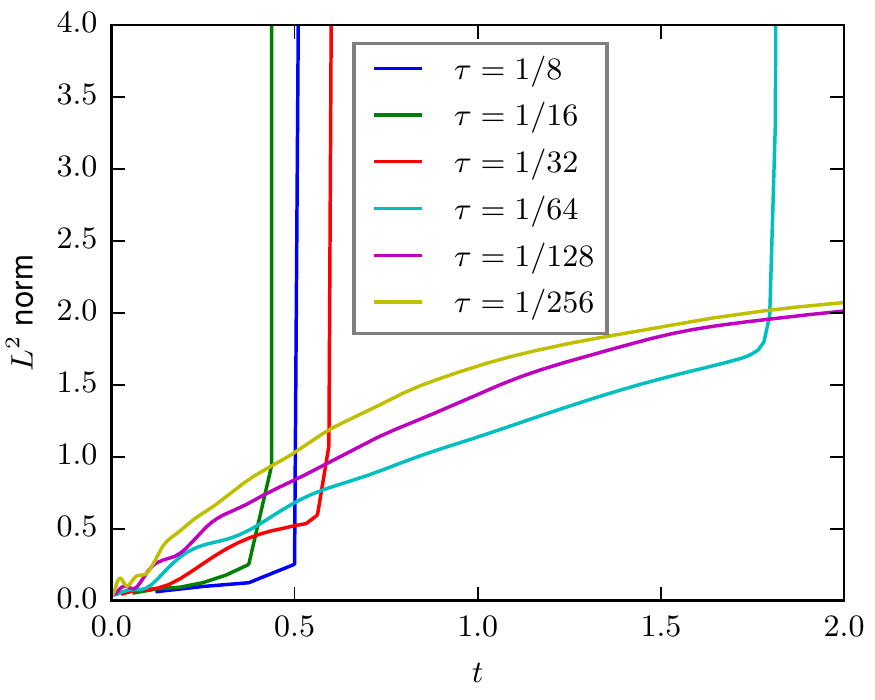} 
\caption{\revision{A Navier-Stokes problem with manufactured solution with $R_e=1000$.  } Evolution of the $L^2$-norm of the velocity over the time interval $[0,2]$ for different setups of time-steps $\tau$. \revision{We use the iGRM method on $40\times 40$ mesh, 
with trial $U_h = \left(\Spl^{3,2}_0\right)^2 \times \SplO^{3,2}$ and test $V_h = \left(\Spl^{4,2}_0\right)^2 \times \SplO^{4,2}$.}}
\label{fig:NScfl}
\end{center}
\end{figure}

\section{Conclusions}
\label{sec:conc}

We applied residual minimization procedures \revision{to non-stationary Stokes and Navier-Stokes problems} in simple, functional settings. We experiment with different polynomial order and continuity of both trial and test spaces.
For non-stationary Stokes and Navier-Stokes problems, we employed the splitting scheme originally proposed in \cite{Minev} for finite-difference simulations. We applied the residual minimization (RM) in every time step for stabilization. We obtained a linear computational cost solver for non-stationary Stokes and Navier-Stokes problems combining IGA and RM. We tested our method on cavity flow problems and problems having manufactured solutions.
\revision{We show that for high Reynolds number, the residual minimization method is required to stabilize the Navier-Stokes simulations.}
The RM method allowed to increase the Reynolds number for which the non-stationary schemes are stable.
\revision{We compared different combinations of trial and test spaces and how they affect the numerical accuracy of the Navier-Stokes problem.
We conclude that for the problem considered in this paper, the higher continuity spaces reduce the computational cost without compromising the solution's accuracy.
However, the standard finite element spaces can be augmented with $hp$-adaptivity \cite{t8}, resulting in exponential convergence of the numerical error with respect to the mesh size. In this paper, however, we advocated an alternative approach, where we employ linear computational cost alternating direction implicit solver on a tensor product grids.}
Future work may include developing an iterative algorithm based on the ideas presented in~\cite{Method1,Method2}.
We will also target the parallelization of the method as in~\cite{parallel1}, possibly using the decomposition of the solver algorithm into basic undividable tasks \cite{parallel2,parallel3}.
\inblue{It is also worth investigating if residual minimization can be applied in specific space-time formulations, preserving the Kronecker product structure and the solver's linear computational cost. In particular, the residual minimization method in space-time may require space-time norms if we intend to stabilize the problem in the space-time setup.}
Our future work will also extend this method to Maxwell problems~\cite{ maxwell1,maxwell2,maxwell3}.

\subsection*{Acknowledgments}
The work of Maciej Paszy\'nski and Marcin \L{}o\'s, and the visit of Judit Mu\~noz-Matute and Ignacio Muga in Krak\'ow is supported by National Science Centre, Poland grant no. 2017/26/M/ ST1/ 00281.
\inblue{Ignacio Muga has received funding from the European Union's Horizon 2020 research and innovation programme under the Marie Sklodowska-Curie grant agreement No 777778 (MATHROCKS). 
The authors also want to thank the Chilean Conicyt project PAI80160025, and the unknown reviewers for their insightful suggestions, which helped us to improve on this version of the manuscript.}

\end{document}